\documentclass[final,leqno]{siamltex}

\usepackage{graphicx}   
\usepackage[numbers]{natbib}
\usepackage{amssymb}
\usepackage{amsmath}
\usepackage{enumerate}
\usepackage{epstopdf}
\usepackage{undertilde}

\title{Feasible-Side Global Convergence in Experimental Optimization}

\author{G. A. Bunin\thanks{Laboratoire d'Automatique, Ecole Polytechnique F\'ed\'erale de Lausanne, 
        Lausanne, Switzerland, 
        1015 ({\tt gene.a.bunin@ccapprox.info}).} \and G. Fran\c cois \and D. Bonvin}

\begin{document}

\maketitle

\begin{abstract}
We propose a canonical form of the experimental optimization problem and review the state-of-the-art methods to solve it. As guarantees of global convergence to an optimal point via only feasible iterates are absent in these methods, we present a set of sufficient conditions to enforce this behavior. 
\end{abstract}



\pagestyle{myheadings}
\thispagestyle{plain}
\markboth{G. A. BUNIN, G. FRAN\c COIS, D. BONVIN}{FEASIBLE-SIDE GLOBAL CONVERGENCE IN EXPERIMENTAL OPTIMIZATION}

\section{The Experimental Optimization Problem}

Consider the problem

\begin{equation}\label{eq:mainprob}
\begin{array}{rll}
\mathop {{\rm{minimize}}}\limits_{\bf{u}} & \phi_p ({\bf{u}}) & \\
{\rm{subject}}\hspace{1mm}{\rm{to}} & g_{p,j}({\bf{u}}) \leq 0, & j = 1,...,n_{g_p} \\
 & g_{j}({\bf{u}}) \leq 0, & j = 1,...,n_{g} \\
& {\bf{u}}^L \preceq {\bf u} \preceq {\bf u}^U, &
\end{array}
\end{equation}

\noindent where ${\bf u} \in \mathbb{R}^{n_u}$ are independent decision variables subject to the bounds ${\bf u}^L$ and ${\bf u}^U$ -- the curly brackets ($\prec$) denoting componentwise inequality -- and $\phi, g : \mathbb{R}^{n_u} \rightarrow \mathbb{R}$ are cost and constraint functions, respectively.

The characteristic element of (\ref{eq:mainprob}) is the presence of \emph{experimental} functions, denoted by the subscript $p$ (for ``plant''), which \emph{may only be evaluated by conducting an experiment for a given choice of ${\bf u}$ and whose values cannot be known otherwise}. In this work, the term ``experiment'' will be employed to denote a repeatable but expensive task, where ``repeatable'' means that carrying out the task once with the variables ${\bf u}_a$ and again with ${\bf u}_b$ will yield identical results if ${\bf u}_a = {\bf u}_b$, while ``expensive'' implies that carrying out the task either is financially costly (e.g., machining a very expensive space shuttle component), requires a lot of time (e.g., simulating one day of traffic behavior for a large metropolis), or may only be done very infrequently (e.g., producing a large batch of a pharmaceutical compound once every three months). Of course, such expenses are not mutually exclusive and may also occur together. By contrast, the constraints without the $p$ subscript indicate numerical functions that can be easily evaluated for any given ${\bf u}$ without requiring any experiments.

Consequently, we will refer to (\ref{eq:mainprob}) as an \emph{experimental optimization} problem. The first formal studies on methodologically solving such problems may be traced back to the 1940s, 50s, and 60s, with the works of Hotelling \citep{Hotelling1941}, Box \citep{Box:51,Box:69}, Brooks \citep{Brooks1959,Brooks1961}, and Spendley \emph{et al.} \citep{Spendley1962} essentially representing the foundations of this field. The methods that came out of these works -- namely, those of (experimental) steepest ascent, evolutionary operation, response-surface modeling, and the simplex algorithm -- have remained popular to the present day and are still employed in a number of diverse applications \citep{Guerra1999,Holmes2003,Guvenc2004,He2005,Bas2007,Bezerra2008,Myers2009}. Additionally, there are entire fields of research dedicated to solving problems that may be cast in the form of (\ref{eq:mainprob}). We cite, as some examples that we have encountered:

\begin{itemize}
\item steady-state optimization \cite{Chen1987,Fatora1992,Naysmith1995,Cheng2004,Brdys2005,Gao2005,Flemming2007,Engell2007,Tatjewski2008} (other names: ``real-time optimization'', ``on-line optimization'', ``measurement-based optimization'', ``set-point optimization'', ``multilayer optimizing control''),
\item optimization of a dynamic profile in a batch process \cite{Srinivasan:03,Srinivasan:03a,Francois2005,Kadam:07,Georgakis:09,Costello:11} (other names: ``run-to-run/batch-to-batch/cycle-to-cycle optimization'', ``dynamic real-time optimization'', ``dynamic measurement-based optimization'', ``(dynamic) optimization of batch processes''),
\item iterative, or run-to-run, controller tuning/design \cite{Hjalmarsson:98,Karimi:04,Krstic:06,Magni2009,Bunin2013p},
\item numerical optimization with expensive function evaluations \cite{Elster1997,Jones1998,Conn2009,Billups2013}.
\end{itemize}

Despite a sufficiently large body of literature, it remains the case that there still exists no theoretically rigorous framework to guarantee that one actually solves Problem (\ref{eq:mainprob}) \emph{reliably}. By ``reliable'', we mean that:

\begin{itemize}
\item The algorithm used to solve (\ref{eq:mainprob}) generates a sequence of experimental iterates ${\bf u}_0,{\bf u}_1,{\bf u}_2,...$ that converge to a set ${\bf u}^*$ that is locally optimal in some sense. This is important since an algorithm that converges to a suboptimal value may be of limited use in practice, especially if one does not know in advance what the potential suboptimality is.
\item The algorithm used to solve (\ref{eq:mainprob}) generates a sequence of experimental iterates ${\bf u}_0,{\bf u}_1,{\bf u}_2,...$ that satisfy all of the problem constraints, \emph{including the experimental constraints} $g_p$, at \emph{every experiment}. This point -- well worth the emphasis -- represents another major characteristic of the experimental optimization problem, as running an experiment that violates some constraint(s) could potentially endanger personnel, promote a hazardous environment, or cause permanent equipment damage. While such test points would simply be ``discarded'' in a numerical optimization framework with only a feasible subsequence considered in the analysis, in experimental optimization they may not be, and so a reliable scheme must only generate experimental iterates guaranteed to satisfy the constraints \emph{always}.
\end{itemize}
 
The main contribution of this work is to build the foundations of such a framework by presenting a set of conditions sufficient to guarantee feasible-side global convergence. So as to maintain as much of the generality of (\ref{eq:mainprob}) as possible, only the following two assumptions are made on the nature of the problem:

\begin{enumerate}[]
\item {\bf{A1}}: The functions $\phi_p$, $g_p$, and $g$ are twice continuously differentiable ($C^2$) on an open set containing the experimental space $\mathcal{I} = \{ {\bf u} : {\bf u}^L \preceq {\bf u} \preceq {\bf u}^U \}$.
\item {\bf{A2}}: The initial experimental iterate, ${\bf{u}}_0$, is strictly feasible with respect to the experimental constraints ($g_{p,j} ({\bf{u}}_0) < 0, \; \forall j = 1,...,n_{g_p}$), feasible with respect to the numerical constraints ($g_{j} ({\bf{u}}_0) \leq 0, \; \forall j = 1,...,n_{g}$), and lies in the experimental space (${\bf u}_0 \in \mathcal{I}$).
\end{enumerate}

\noindent Here, Assumption A1 is required to obtain general-purpose Lipschitz bounds \cite{Bunin:Lip,Cartis2013}, which follow from the existence and boundedness of the first and second derivatives on the experimental space and are crucial throughout the analysis. Assumption A2 is needed due to the feasible-side requirement, but should not be restrictive since one would not expect to begin optimization with a hazardous experiment at ${\bf u}_0$.

Following a brief review of the state of the art in experimental optimization in Section \ref{sec:review}, we proceed to summarize the sufficient conditions for feasible-side global convergence in Section \ref{sec:scfo}. Because these conditions cannot, in general, be expected to hold innately, we propose a numerical implementation method to enforce them in Section \ref{sec:implement} and prove, in Section 5, that such an implementation must generate a sequence of feasible experimental iterates that converge arbitrarily close to a Fritz John (FJ) stationary point while monotonically decreasing the cost function value. As this implementation requires certain tunable parameters to be sufficiently small to guarantee optimality, an adaptive scheme to choose these parameters without significantly compromising convergence speed is proposed in Section 6. An illustration of the method is then given in Section \ref{sec:example}, after which we conclude the paper with some remarks on the practical usefulness of the method, what has been achieved, and what is planned for the future.

\section{Viable Algorithms and Guarantees}
\label{sec:review}

In reviewing the existing methods capable of solving (\ref{eq:mainprob}), we choose to classify the relevant algorithms as being either model-based or model-free. As our focus in this work is particularly on theoretical feasibility and optimality \emph{guarantees}, we review only these aspects and do not treat issues like convergence speed or ease of implementation. 

\subsection{Model-based Algorithms}

We will consider as ``model-based'' any algorithm that attempts to solve (\ref{eq:mainprob}) by iteratively solving its parameterized model approximation

\begin{equation}\label{eq:mainprobmod}
\begin{array}{rll}
\mathop {{\rm{minimize}}}\limits_{\bf{u}} & \phi_{\hat p} ({\bf{u}},{\boldsymbol \theta}) & \\
{\rm{subject}}\hspace{1mm}{\rm{to}} & g_{{\hat p},j}({\bf{u}},{\boldsymbol \theta}) \leq 0, & j = 1,...,n_{g_p} \\
 & g_{j}({\bf{u}}) \leq 0, & j = 1,...,n_{g} \\
& {\bf{u}}^L \preceq {\bf u} \preceq {\bf u}^U, &
\end{array}
\end{equation}

\noindent with $\hat p$ denoting the models of the experimental functions and ${\boldsymbol \theta}$ the parameters. As one might expect, what essentially differentiates one model-based algorithm from another are the nature of ${\boldsymbol \theta}$ and how these parameters are estimated.

A popular model-based framework is that of trust-region methods \cite{Alexandrov1997,Elster1997,Conn2000,Wright2004,Peng2006,Conn2009,Billups2013}, where the model functions are usually quadratic and ${\boldsymbol \theta}$ are the coefficients of the models. While generally a tool used for numerical optimization, these methods could certainly be applied to experimental problems \citep{Elster1997,Billups2013} and would proceed by constructing an approximate model based on local measurements, optimizing this model over a local trust region, and, depending on whether the new solution led to improvement or not, either increasing/decreasing the trust region and, when necessary, reconstructing the model. Global convergence results are available for the unconstrained case \citep{Conn2009} and the case with ``simple'' constraints \citep{Conn2000}, with a penalty function method to force global convergence for the general constrained problem (\ref{eq:mainprob}) also a possibility \citep{Liuzzi2010}. However, feasible-side convergence in the experimental-setting sense is absent in these methods, and while algorithms with only ``feasible iterates'' do exist \cite{Wright2004,Peng2006}, what is considered a single iteration in these algorithms is not the application of a single ${\bf u}$ but a series of such ${\bf u}$, with only the guarantee that the final choice is feasible (the infeasible ${\bf u}$ being discarded). Penalty-function methods like those proposed in \citep{Liuzzi2010} guarantee feasibility upon asymptotic convergence, but cannot guarantee that every ${\bf u}$ applied satisfy the constraints of (\ref{eq:mainprob}).

In the engineering context, a popular model-based technique is that of identification followed by optimization (e.g., the so-called ``two-phase'' or ``two-step'' approach \cite{Chen1987,Jang1987} in the chemical engineering literature or the ``indirect'' tuning method \cite{Landau2011} in adaptive control), where a first-principles parametric model of the system under consideration is updated by re-estimating the parameters following the acquisition of new data, with the updated model then being optimized to yield a new optimal target. Despite its fairly wide acceptance in industry \cite{Fatora1992,Quelhas:12}, the prominent theoretical weakness of this method is its inability to adapt to \emph{structural} model errors -- also known as ``plant-model mismatch'' \cite{Krishnan1992} -- as convergence to a stationary point can only be guaranteed if (\ref{eq:mainprob}) and (\ref{eq:mainprobmod}) have the same stationarity conditions \cite{Biegler1985}. While trust-region methods may avoid this by simply shrinking the trust region until the quadratic model becomes a suitable approximation of the structure of the true problem, the identification-optimization technique usually maintains the same optimization domain and only updates the parameters. As such, this method comes without guarantees of convergence to an optimum. Furthermore, feasibility guarantees are also absent. While some work has attempted to guarantee that the constraints of (\ref{eq:mainprob}) be satisfied robustly for various stochastically distributed values of ${\boldsymbol \theta}$ \cite{Zhang2002,Li2008}, such approaches are not generally robust as they inherently assume the availability of a model where all of the modeling errors are parametric. A more practical alternative is to add safety ``back-offs'' to the constraints \cite{Loeblein1998,Loeblein1999,Govatsmack2005,Quelhas:12}, but this is only an \emph{ad hoc} solution that ultimately does not guarantee feasibility while introducing suboptimality into the solution of (\ref{eq:mainprob}).  

A proposed alternative to the identification-optimization approach is that of Karush-Kuhn-Tucker (KKT) correction methods, which are based on the original work of Roberts \cite{Roberts1978} and are known in the literature as ``ISOPE'' (integrated system optimization and parameter estimation) \cite{Brdys2005,Gao2005,Xu2008} or as ``modifier adaptation'' \cite{Marchetti2009a,Rodger2010}. The motivation behind these methods is precisely the convergence problem of the identification-optimization approach, which is avoided by adding first-order correction terms to the model functions. Unlike the identification-optimization approach, one need not update the inherent model parameters but only the correction terms (these latter thus take the role of ${\boldsymbol \theta}$ in (\ref{eq:mainprobmod})), which ensure that the structures of (\ref{eq:mainprob}) and (\ref{eq:mainprobmod}) match locally to first order. As this is sufficient for stationarity, it follows that such an approach guarantees both feasibility and convergence to a stationary point \emph{if the scheme converges}. However, the guarantee that the scheme converge globally is still a topic of research. In their monograph, Brdys and Tatjewski \cite{Brdys2005} provide sufficient conditions for the special case where no experimental constraints are present and where the numerical constraints $g$ are convex. A general conceptual sufficient condition based on fixed-point theory has been proposed in \cite{Faulwasser2014}. The recent work in \cite{BuninMATR} has proposed adding a trust-region ``wrapping'' to these methods so that the global convergence properties of trust-region schemes may be obtained, with the same idea, though in a different context, also presented in \cite{Biegler2014}. As none of these approaches guarantee feasible-side iterates in the presence of experimental constraints, the work in \cite{Bunin2011} has proposed using a filter in the adaptation between experiments as a means of preserving feasibility, but in doing so sacrificed a number of important convergence and optimality properties.

Finally, a very popular and simple approach is that of response-surface modeling \citep{Bas2007,Bezerra2008,Myers2009}, where a set of prescribed and optimally designed experiments are carried out to construct a data-driven model of (\ref{eq:mainprob}). In many applications, either a central composite or a Box-Behnken \cite{Ferreira2007} experimental design is used and a quadratic model is built \cite{Myers2009}, but more advanced alternatives may also be pursued \cite{Jones1998}. While very similar in nature to trust-region methods, the fundamental difference lies in the fact that response-surface modeling generally attempts to construct a \emph{global} approximation of the problem over the entire experimental space. For models of fixed structure, this obviously presents the same theoretical disadvantage as the identification-optimization approach discussed earlier -- i.e., if the chosen response models are quadratic and the experimental functions are not, it is impossible to prove that any choice of model coefficients (${\boldsymbol \theta}$) will yield a stationary point of (\ref{eq:mainprob}). An adaptive structure that is consistent in the sense that it is able to globally approximate the true function arbitrarily well as more data is obtained would, however, be able to guarantee global convergence, the DACE model used by Jones {\it et al.} \citep{Jones1998} being one example. However, it is once more feasibility that poses a major issue, as response-surface modeling does not take the experimental constraints into account when choosing what experiments to run while constructing the model.

\subsection{Model-free Algorithms}

By ``model-free'' we refer to algorithms that do not start with, nor attempt to construct, a model of (\ref{eq:mainprob}).

Direct search methods are, without question, the best established of the algorithms that belong to this class, and are generally known for their robust convergence properties \cite{Lewis2000}. In addition to the classic approaches like the pattern search of Hooke and Jeeves \cite{Hooke1961}, methods like the simplex approach \cite{Spendley1962,Nelder1965} and evolutionary operation \cite{Box:69} have been accepted in both laboratory and industrial experimental settings \cite{King1975,Wade1990,Holmes2003,Biercuk2009} and could be easily modified to obtain global convergence for the unconstrained case (using step-size reduction techniques and the like -- see, e.g., \cite{Conn2009}), with the constrained case following suit via standard (e.g., penalty-function) techniques \cite{Liuzzi2010}. To the best of our knowledge, no guarantees regarding purely feasible-side experimental iterates are available, however. While many of these methods can enforce feasibility by changing the step size in the often-employed line search, there is no guarantee that all points tested during a given line search will satisfy the constraints of (\ref{eq:mainprob}) as required in the experimental setting -- the infeasible points, again, simply being discarded.

A natural model-free alternative to direct search methods is the approximate gradient-descent method \citep{Brooks1961,DelCastillo1997}, where the gradient of the cost function is estimated and then used in a line search. The standard methods of estimating the gradient include taking finite differences or regressing the available data, although in contexts where Problem (\ref{eq:mainprob}) involves a transient stage one may also use the dynamic data obtained during an experiment to estimate the gradient of the inherently \emph{static} experimental functions -- see, e.g., \cite{Hjalmarsson:98} or \cite{Bamberger1978}. Global convergence may be proven to within a certain tolerance that depends on the error of the estimate \cite{Gratton2011} in the unconstrained case, and a natural extension to the constrained case using the penalty approach is possible \cite{Liuzzi2010}. No guarantee of feasibility is available during the line search.

\section{Sufficient Conditions for Feasible-Side Global Convergence}
\label{sec:scfo}

Letting $k$ denote the experiment counter (the experimental iteration), we proceed to state a set of recursive conditions that, when satisfied by every future experimental iterate ${\bf u}_{k+1}$ given the information obtained from the current experiment at ${\bf u}_k$, are sufficient to guarantee feasible-side global convergence:

\begin{equation}\label{eq:SCFO1}
g_{p,j}({\bf u}_k) + \displaystyle \mathop {\sum} \limits_{i=1}^{n_u} \kappa_{p,ji} | u_{k+1,i} - u_{k,i} | \leq 0, \;\; \forall j = 1,...,n_{g_p},
\end{equation}

\begin{equation}\label{eq:SCFO2}
g_{j}({\bf u}_{k+1}) \leq 0, \;\; \forall j = 1,...,n_g,
\end{equation}

\begin{equation}\label{eq:SCFO3}
{\bf u}^L \preceq {\bf u}_{k+1} \preceq {\bf u}^U,
\end{equation}

\begin{equation}\label{eq:SCFO4}
\nabla g_{p,j} ({\bf u}_k)^T ({\bf u}_{k+1} - {\bf u}_k) < 0, \;\; \forall j: g_{p,j} ({\bf u}_k) \approx 0,
\end{equation}

\begin{equation}\label{eq:SCFO5}
\nabla g_{j} ({\bf u}_k)^T ({\bf u}_{k+1} - {\bf u}_k) < 0, \;\; \forall j: g_{j} ({\bf u}_k) \approx 0,
\end{equation}

\begin{equation}\label{eq:SCFO6}
\nabla \phi_{p} ({\bf u}_k)^T ({\bf u}_{k+1} - {\bf u}_k) < 0,
\end{equation}

\begin{equation}\label{eq:SCFO7}
\nabla \phi_p ({\bf u}_k)^T ({\bf u}_{k+1} - {\bf u}_k) \displaystyle+ \frac{1}{2} \mathop {\sum} \limits_{i_1=1}^{n_u} \mathop {\sum} \limits_{i_2=1}^{n_u} M_{\phi,i_1 i_2} | (u_{k+1,i_1} - u_{k,i_1})(u_{k+1,i_2} - u_{k,i_2}) | \leq 0,
\end{equation}

\noindent where $\kappa_{p,ji}$ and $M_{\phi,i_1 i_2}$ are used to denote the univariate Lipschitz constants of the experimental constraint functions and the Lipschitz constants of the cost function derivatives, respectively, and are defined implicitly as

\begin{equation}\label{eq:lipcon}
- \kappa_{p,ji} < \frac{\partial g_{p,j}}{\partial u_i} \Big |_{\bf u} < \kappa_{p,ji}, \;\; \forall {\bf u} \in \mathcal{I},
\end{equation}

\begin{equation}\label{eq:lipcon2}
-M_{\phi,i_1 i_2} < \frac{\partial^2 \phi_p}{\partial u_{i_2} \partial u_{i_1} } \Big |_{\bf u} < M_{\phi,i_1 i_2}, \;\; \forall {\bf u} \in \mathcal{I}.
\end{equation}

\noindent The notation $u_{k,1},u_{k,2},...,u_{k,n_u}$ is used to denote the different elements of ${\bf u}_k$.

As given, Conditions (\ref{eq:SCFO1})-(\ref{eq:SCFO7}) are not implementable numerically due to the strict inequalities in (\ref{eq:SCFO4})-(\ref{eq:SCFO6}) and the approximate equalities in (\ref{eq:SCFO4})-(\ref{eq:SCFO5}), which may thus be approximated by some numerical tolerances. Deferring this discussion to Section 4, we first give a qualitative overview of the different conditions and their \emph{raison d'{\^e}tre}.

\subsection{The Feasibility Conditions (\ref{eq:SCFO1})-(\ref{eq:SCFO3})}

Start by noting that Conditions (\ref{eq:SCFO2}) and (\ref{eq:SCFO3}) are simply the numerical and bound constraints of the main problem (\ref{eq:mainprob}). As both of these are easily checked for any ${\bf u}_{k+1}$, there is no need to transform them out of their original form as they are already tractable.

The same is not true for the experimental constraint functions $g_p$, for which the trivial sufficient condition, $g_{p,j} ({\bf u}_{k+1}) \leq 0$, is intractable as the function $g_{p,j}$ is unknown. One may sidestep this difficulty by employing the upper bound

\begin{equation}\label{eq:bound1U}
g_{p,j} ({\bf u}_{k+1}) \leq g_{p,j} ({\bf u}_k) + \displaystyle \sum_{i=1}^{n_u} \kappa_{p,ji} | u_{k+1,i} - u_{k,i} |,
\end{equation}

\noindent the derivation of which may be found in \cite{Bunin:Lip}. Clearly, enforcing the right-hand side of (\ref{eq:bound1U}) to be non-positive also enforces $g_{p,j} ({\bf u}_{k+1}) \leq 0$, and this then yields (\ref{eq:SCFO1}). While it may appear that we have ``cheated'' in replacing one intractable condition by another -- in principle, one cannot know the Lipschitz constants $\kappa_{p,ji}$ since one does not know $g_{p,j}$ -- the latter condition only requires having sufficiently conservative values of the Lipschitz constants while the former requires knowing the entire function. The reader is referred to \citep{Bunin:SCFOImp} for the different ways to estimate these constants in practice, and to \citep{SCFOug} and \citep{Bunin2013p} for examples of their successful use in implementation.

One should remark that whenever $g_{p,j} ({\bf u}_{k}) < 0$, there always exists a ${\bf u}_{k+1}$ sufficiently close to ${\bf u}_{k}$ so as to satisfy (\ref{eq:SCFO1}) with ${\bf u}_{k+1} \neq {\bf u}_k$. Also, since the bound (\ref{eq:bound1U}) holds with \emph{strict} inequality whenever ${\bf u}_{k+1} \neq {\bf u}_k$ \cite{Bunin:Lip}, satisfying (\ref{eq:SCFO1}) in turn implies $g_{p,j} ({\bf u}_{k+1}) < 0$. That strict feasibility for the experimental constraints is maintained is thus proven trivially by induction, starting with the base case at ${\bf u}_0$, with $g_{p,j} ({\bf u}_0) < 0$ following from Assumption A2.

\subsection{The Strict Monotonic Improvement Conditions (\ref{eq:SCFO6}) and (\ref{eq:SCFO7})}

Condition (\ref{eq:SCFO6}) is a standard local descent condition and enforces that the next experimental iterate lie in the strict descent halfspace of the cost function. In doing so, an improvement in the cost value is guaranteed provided that ${\bf u}_{k+1}$ is sufficiently close to ${\bf u}_k$. This is achieved with Condition (\ref{eq:SCFO7}), which follows from the upper bound

\begin{equation}\label{eq:bound2U}
\begin{array}{l}
\phi_p({\bf u}_{k+1}) - \phi_p({\bf u}_{k}) \leq \nabla \phi_p({\bf u}_{k})^T ({\bf u}_{k+1} - {\bf u}_{k}) + \vspace{1mm} \\
\hspace{10mm}\displaystyle \frac{1}{2} \sum_{i_1=1}^{n_u} \sum_{i_2=1}^{n_u} M_{\phi,i_1 i_2} | (u_{k+1,i_1} - u_{k,i_1})(u_{k+1,i_2} - u_{k,i_2}) |,
\end{array}
\end{equation}

\noindent where forcing the right-hand side to be non-positive implies $\phi_p({\bf u}_{k+1}) - \phi_p({\bf u}_{k}) < 0$ for ${\bf u}_{k+1} \neq {\bf u}_k$ \cite{Bunin:Lip}. This is attainable for a ${\bf u}_{k+1}$ sufficiently close to ${\bf u}_k$ since the linear term of the expression is forced to be strictly negative by (\ref{eq:SCFO6}) and overwhelms the quadratic term locally.

\subsection{The Projection Conditions (\ref{eq:SCFO3})-(\ref{eq:SCFO5})}

While enforcing both feasibility and monotonic improvement may appear to promise convergence to a point where no locally feasible, cost-descent direction exists (i.e., to a stationary point), it should be clear that something is missing. This ``something'' is the regulation of the implicit distance between ${\bf u}_{k+1}$ and ${\bf u}_k$, which, while needing to be sufficiently small to satisfy (\ref{eq:SCFO1})-(\ref{eq:SCFO3}) and (\ref{eq:SCFO7}), cannot become too small since this leads to $\phi_{p} ({\bf u}_{k+1}) - \phi_{p} ({\bf u}_k) \rightarrow 0$ and precludes global convergence to a stationary point. We give a simple illustration of an algorithm that satisfies Conditions (\ref{eq:SCFO1})-(\ref{eq:SCFO3}) and (\ref{eq:SCFO6})-(\ref{eq:SCFO7}) but fails to converge to a stationary point in Fig. \ref{fig:feasbreak}. Here, a gradient-descent algorithm continually takes steps in the (linear) cost-descent direction while maintaining feasibility. While the iterates ${\bf u}_0,...,{\bf u}_\infty$ may indeed be strictly monotonically decreasing in cost and remain strictly feasible, it is easily seen that the algorithm approaches ${\bf u}_\infty$ rather than the optimum ${\bf u}^*$. This is essentially due to the algorithm taking smaller and smaller steps as it approaches $g_{p,1}$, a behavior that is forced by Condition (\ref{eq:SCFO1}) and needed to guarantee that the constraint is not violated.

\begin{figure}
\begin{center}
\includegraphics[width=7cm]{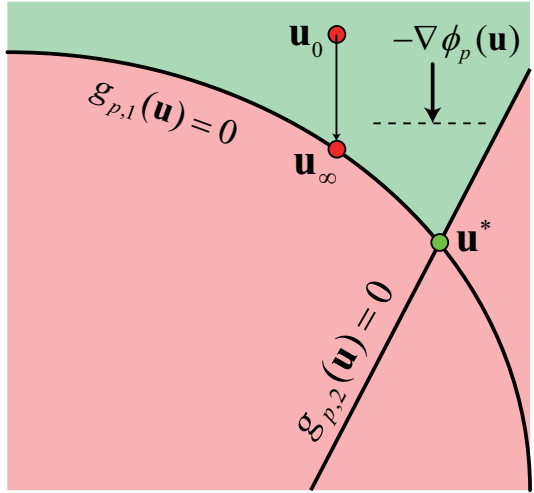}
\caption{Illustration of an algorithm that fails to converge to a stationary point of (\ref{eq:mainprob}) by virtue of not satisfying the projection conditions.}
\label{fig:feasbreak}
\end{center}
\end{figure}

This motivates the idea of projection and is analogous, albeit not identical, to what is often employed in numerical optimization schemes that project search points onto the feasible space \cite{Calamai1987,Conn2000,Wright2004}. In the context of experimental optimization and the conditions proposed here, the motivation for projecting is to stay sufficiently deep inside the feasible space so as to avoid the case where either Condition (\ref{eq:SCFO1}), (\ref{eq:SCFO2}), or (\ref{eq:SCFO3}) lead to ${\bf u}_{k+1} \rightarrow {\bf u}_k$ (as opposed to the numerical context, where projection is done to return infeasible iterates to the feasible region).

This is achieved with Conditions (\ref{eq:SCFO4}) and (\ref{eq:SCFO5}) for the experimental and numerical inequality constraints, respectively, where we force ${\bf u}_{k+1}$ to lie in the local descent halfspaces of any constraints that are close to active. The same is achieved by Condition (\ref{eq:SCFO3}), which does double duty both as a feasibility and a projection condition. This is easily seen if one considers one of the bounds to be active -- taking, e.g., $u_{k,1} = u^U_1$ leads to the condition $u_{k+1,1} - u_{k,1} \leq 0$, which is the projection condition seeing as the gradient of this constraint is a vector of zeros with the sole first element equal to 1. Applying projection to the example in Fig. \ref{fig:feasbreak}, one could visualize the experimental iterates ``sliding'' along $g_{p,1}$ once it became close to active, before eventually converging in the neighborhood of ${\bf u}^*$. 

\section{Basic Numerical Implementation of the Conditions}
\label{sec:implement}

As already mentioned, some level of approximation is needed to replace the strict inequalities and the approximate equalities in (\ref{eq:SCFO4})-(\ref{eq:SCFO6}) if one is to implement these conditions numerically. Letting ${\boldsymbol \epsilon}_{p}, {\boldsymbol \epsilon}, {\boldsymbol \delta}_{g_p}, {\boldsymbol \delta}_{g}, \delta_\phi \succ {\bf 0}$ denote the appropriate \emph{projection parameters}, with $\epsilon_{p,j}$, $\epsilon_{j}$, $\delta_{g_p,j}$, and $\delta_{g,j}$ denoting individual elements, we state the numerically implementable versions of (\ref{eq:SCFO4})-(\ref{eq:SCFO6}):

\begin{equation}\label{eq:SCFO4i}
\nabla g_{p,j} ({\bf u}_k)^T ({\bf u}_{k+1} - {\bf u}_k) \leq -\delta_{g_p,j}, \;\; \forall j: g_{p,j} ({\bf u}_k) \geq -\epsilon_{p,j},
\end{equation}

\begin{equation}\label{eq:SCFO5i}
\nabla g_{j} ({\bf u}_k)^T ({\bf u}_{k+1} - {\bf u}_k) \leq -\delta_{g,j}, \;\; \forall j: g_{j} ({\bf u}_k) \geq -\epsilon_{j},
\end{equation}

\begin{equation}\label{eq:SCFO6i}
\nabla \phi_{p} ({\bf u}_k)^T ({\bf u}_{k+1} - {\bf u}_k) \leq -\delta_\phi,
\end{equation}

At this point, one is faced with the question of how to find an experimental iterate ${\bf u}_{k+1}$ that would enforce the sufficient conditions. As all of the conditions are stated as inequality constraints on ${\bf u}_{k+1}$, the natural approach would be to project a point onto this inequality set:

\begin{equation}\label{eq:projgen}
\begin{array}{rrl}
{\bf u}_{k+1} := & {\rm arg}\mathop {{\rm{minimize}}}\limits_{{\bf u}} & \| {\bf u} - {\bf u}_{k+1}^* \|_2^2 \\
& {\rm{subject}}\hspace{1mm}{\rm{to}} & (\ref{eq:SCFO1}){\rm -}(\ref{eq:SCFO3}),(\ref{eq:SCFO4i}){\rm -}(\ref{eq:SCFO6i}),(\ref{eq:SCFO7}),
\end{array}
\end{equation} 

\noindent with the substitution ${\bf u}_{k+1} \rightarrow {\bf u}$ implicit in (\ref{eq:SCFO1})-(\ref{eq:SCFO3}),(\ref{eq:SCFO4i})-(\ref{eq:SCFO6i}), and (\ref{eq:SCFO7}). Here, we use ${\bf u}_{k+1}^*$ to denote a ``target'' that we believe would lead to better performance but may not necessarily satisfy the sufficient conditions. In practice, this target may be provided by any of the aforementioned approaches of Section \ref{sec:review} -- i.e., it is suggested that the implementation of the conditions be coupled with another experimental optimization method, although this is purely optional from the theoretical perspective, since any arbitrary ${\bf u}_{k+1}^*$ could be chosen.

While the approach of (\ref{eq:projgen}) is perfectly valid, we do not choose it as our preferred strategy for the reason that potential numerical issues could come up while solving (\ref{eq:projgen}). This is because neither Condition (\ref{eq:SCFO2}) nor Condition (\ref{eq:SCFO7}) are guaranteed to yield convex constraints in the projection problem -- specifically, (\ref{eq:SCFO2}) yields convex constraints only when all of the numerical constraints are convex, while the requirements for (\ref{eq:SCFO7}) to be convex are more involved. In the absence of such restrictions, reliably obtaining a feasible solution to (\ref{eq:projgen}) may be difficult to guarantee.

A less elegant but completely tractable approach to numerically implementing the conditions is that of \emph{projecting and filtering}, where the projection is only carried out with respect to Conditions (\ref{eq:SCFO3})-(\ref{eq:SCFO6}):

\begin{equation}\label{eq:proj}
\begin{array}{rll}
\bar {\bf u}_{k+1}^* := {\rm arg} \mathop {\rm minimize}\limits_{{\bf u}} & \| {\bf u} - {\bf u}_{k+1}^* \|_2^2 & \\
{\rm{subject}}\;{\rm{to}} & \nabla g_{p,j} ({\bf u}_k)^T ({\bf u} - {\bf u}_k) \leq -\delta_{g_p,j}, & \hspace{-2mm} \forall j : g_{p,j}({\bf u}_k) \geq -\epsilon_{p,j} \\
 & \nabla g_{j} ({\bf u}_k)^T ({\bf u} - {\bf u}_k) \leq -\delta_{g,j}, & \hspace{-2mm} \forall j : g_{j}({\bf u}_k) \geq -\epsilon_{j} \\
 & \nabla \phi_{p} ({\bf u}_k)^T ({\bf u} - {\bf u}_k) \leq -\delta_{\phi} & \\
 & {\bf u}^L \preceq {\bf u} \preceq {\bf u}^U &
\end{array}
\end{equation}

\noindent to yield the \emph{projected target} $\bar {\bf u}_{k+1}^*$. This target is then filtered to give the new input point using the law

\begin{equation}\label{eq:inputfilter}
{\bf u}_{k+1} := {\bf u}_k + K_k \left( \bar {\bf u}_{k+1}^* - {\bf u}_{k} \right),
\end{equation}

\noindent where $K_k \in [0,1]$ denotes the filter gain. Such an implementation essentially allows for Conditions (\ref{eq:SCFO1}), (\ref{eq:SCFO2}), and (\ref{eq:SCFO7}) to be remplaced -- by simply substituting (\ref{eq:inputfilter}) for ${\bf u}_{k+1}$ in these expressions -- with conditions on $K_k$:

\begin{equation}\label{eq:SCFO1i}
g_{p,j}({\bf u}_k) + K_k \displaystyle \mathop {\sum} \limits_{i=1}^{n_u} \kappa_{p,ji}|\bar u_{k+1,i}^* - u_{k,i}| \leq 0, \;\; \forall j = 1,...,n_{g_p},
\end{equation}

\begin{equation}\label{eq:SCFO2i}
g_{j}({\bf u}_k + K_k ( \bar {\bf u}_{k+1}^* - {\bf u}_{k} )) \leq 0, \;\; \forall j = 1,...,n_g,
\end{equation}

\begin{equation}\label{eq:SCFO7i}
\begin{array}{l}
\nabla \phi_p ({\bf u}_k)^T (\bar {\bf u}_{k+1}^* - {\bf u}_k) \\
\hspace{10mm}\displaystyle+ \frac{K_k}{2} \mathop {\sum} \limits_{i_1=1}^{n_u} \mathop {\sum} \limits_{i_2=1}^{n_u} M_{\phi,i_1 i_2} | (\bar u_{k+1,i_1}^* - u_{k,i_1})(\bar u_{k+1,i_2}^* - u_{k,i_2}) | \leq 0,
\end{array}
\end{equation}

\noindent for which the value of $K_k$ is maximized subject to (\ref{eq:SCFO1i}), (\ref{eq:SCFO2i}), and (\ref{eq:SCFO7i}) by carrying out a line search on $K_k \in [0,1]$ -- maximization being the natural choice since larger adaptation steps are usually more favorable and lead to faster convergence/improvement. The intuitive appeal of the project-and-filter approach is that (\ref{eq:proj}) points us in the appropriate locally feasible descent direction while the line search then decides how far in that direction to step. Note that the final ${\bf u}_{k+1}$ obtained in this manner, while not necessarily satisfying (\ref{eq:SCFO4i})-(\ref{eq:SCFO6i}) due to the filtering, will nevertheless satisfy the original conditions (\ref{eq:SCFO4})-(\ref{eq:SCFO6}) due to the convexity of this set and the fact that $K_k \in [0,1]$. The practical appeal of the line search is that it can handle all sorts of nonconvexity in (\ref{eq:SCFO2i}) and (\ref{eq:SCFO7i}).

For the case when Projection (\ref{eq:proj}) is infeasible, we simply set ${\bf u}_{k+1} := {\bf u}_k$. which effectively terminates the algorithm due to the projection again being infeasible at ${\bf u}_{k+1}$, ${\bf u}_{k+2}$, and so on.

\section{Proof of Feasible-Side Global Convergence}
\label{sec:proof}

It should already be clear that the project-and-filter approach will not generate any experimental points ${\bf u}_{k+1}$ that violate the problem constraints, since one starts with a feasible point (Assumption A2) and may always choose $K_k$ sufficiently small to preserve feasibility. However, the global convergence properties of this scheme remain somewhat nontrivial. We now proceed to prove that the scheme, for a fixed choice of ${\boldsymbol \epsilon}_{p}, {\boldsymbol \epsilon}, {\boldsymbol \delta}_{g_p}, {\boldsymbol \delta}_{g}, \delta_\phi \succ {\bf 0}$, generates a set of experimental iterates that preserve feasibility and decrease monotonically in cost function value prior to converging to some point in a finite number of iterations due to the projection (\ref{eq:proj}) becoming infeasible. We then show that the point where such infeasibility occurs approaches an FJ point as ${\boldsymbol \epsilon}_{p}, {\boldsymbol \epsilon}, {\boldsymbol \delta}_{g_p}, {\boldsymbol \delta}_{g}, \delta_\phi \downarrow {\bf 0}$.

\subsection{Guarantee of a Minimal, Strictly Positive Filter Gain $K_k$}

We will start by supposing that the projection (\ref{eq:proj}) has been successful and that a feasible solution has been found for all experimental iterations from 0 to $k$, and will analyze the behavior of $K_k$ when it is set as the maximum value on the interval $[0,1]$ that satisfies Conditions (\ref{eq:SCFO1i}), (\ref{eq:SCFO2i}), and (\ref{eq:SCFO7i}). Our goal with this analysis is to show that $K_k$ \emph{must} stay above a strictly positive minimum value that may be expressed as a function of the problem characteristics (the Lipschitz constants, the size of $\mathcal{I}$), the implementation settings (the projection parameters), and the proximity of the initial point to the constraints. Doing so will allow us to guarantee that the project-and-filter approach does not converge prematurely due to a vanishing $K_k$ \cite{Bunin2011}.

Prior to proving this result, we will first require the following definitions.

\begin{definition}[The strictness coefficient for the Lipschitz constants of the experimental constraint functions]
\label{def:strictness}
Denote by $\tilde \kappa_{p,ji}$ the \emph{nonstrict} Lipschitz constants for the constraint function $g_{p,j}$:

\begin{equation}\label{eq:lipconns}
- \tilde \kappa_{p,ji} \leq \frac{\partial g_{p,j}}{\partial u_i} \Big |_{\bf u} \leq \tilde \kappa_{p,ji}, \;\; \forall {\bf u} \in \mathcal{I}.
\end{equation}

The strictness coefficient for the Lipschitz constants of $g_{p,j}$ is then defined as

\begin{equation}\label{eq:gamma}
\gamma_j = \mathop {\max} \limits_{i = 1,...,n_u} \frac{\tilde \kappa_{p,ji}}{\kappa_{p,ji}} < 1.
\end{equation}

\end{definition}

\begin{definition}[Upper bounds on the worst-case linear growth of the experimental and numerical constraint functions]
\label{def:worstlin}
Let $\overline L_{p,j}$ and $\overline L_{j}$ denote the following upper bounds on the worst-case linear growth terms for the experimental and numerical constraint functions, respectively:

\begin{equation}\label{eq:Kupper}
\overline L_{p,j} = \mathop {\sum} \limits_{i=1}^{n_u} \kappa_{p,ji}(u^U_i - u^L_i) \geq \mathop {\sum} \limits_{i=1}^{n_u} \kappa_{p,ji}|\bar u_{k+1,i}^* - u_{k,i}|,
\end{equation}

\begin{equation}\label{eq:Kupper2}
\overline L_j = \mathop {\sum} \limits_{i=1}^{n_u} \kappa_{ji}(u^U_i - u^L_i) \geq \mathop {\sum} \limits_{i=1}^{n_u} \kappa_{ji}|\bar u_{k+1,i}^* - u_{k,i}|,
\end{equation}

\noindent with $\kappa_{ji}$ defined in the same manner as $\kappa_{p,ji}$ in (\ref{eq:lipcon}):

\begin{equation}\label{eq:lipconnum}
- \kappa_{ji} < \frac{\partial g_{j}}{\partial u_i} \Big |_{\bf u} < \kappa_{ji}, \;\; \forall {\bf u} \in \mathcal{I}.
\end{equation}

\end{definition}

\begin{definition}[Upper bounds on the worst-case quadratic growth of the cost and constraint functions]
\label{def:worstquad}
Let $\overline Q_{\phi}$, $\overline Q_{j}$, and $\overline Q_{p,j}$ denote the following upper bounds on the worst-case quadratic growth terms:

\begin{equation}\label{eq:Mupper1}
\begin{array}{l}
\overline Q_\phi = \displaystyle \mathop {\sum} \limits_{i_1=1}^{n_u} \mathop {\sum} \limits_{i_2=1}^{n_u} M_{\phi,i_1 i_2} (u^U_{i_1} - u^L_{i_1})(u^U_{i_2} - u^L_{i_2}) \geq \\
\hspace{20mm} \displaystyle \mathop {\sum} \limits_{i_1=1}^{n_u} \mathop {\sum} \limits_{i_2=1}^{n_u} M_{\phi,i_1 i_2} | (\bar u_{k+1,i_1}^* - u_{k,i_1})(\bar u_{k+1,i_2}^* - u_{k,i_2}) |,
\end{array}
\end{equation}

\begin{equation}\label{eq:Mupper2}
\begin{array}{l}
\overline Q_j = \displaystyle \mathop {\sum} \limits_{i_1=1}^{n_u} \mathop {\sum} \limits_{i_2=1}^{n_u} M_{g_j,i_1 i_2} (u^U_{i_1} - u^L_{i_1})(u^U_{i_2} - u^L_{i_2}) \geq \\
\hspace{20mm} \displaystyle \mathop {\sum} \limits_{i_1=1}^{n_u} \mathop {\sum} \limits_{i_2=1}^{n_u} M_{g_j,i_1 i_2} | (\bar u_{k+1,i_1}^* - u_{k,i_1})(\bar u_{k+1,i_2}^* - u_{k,i_2}) |,
\end{array}
\end{equation}

\begin{equation}\label{eq:Mupper3}
\begin{array}{l}
\overline Q_{p,j} = \displaystyle \mathop {\sum} \limits_{i_1=1}^{n_u} \mathop {\sum} \limits_{i_2=1}^{n_u} M_{g_{p,j},i_1 i_2} (u^U_{i_1} - u^L_{i_1})(u^U_{i_2} - u^L_{i_2}) \geq \\
\hspace{20mm} \displaystyle \mathop {\sum} \limits_{i_1=1}^{n_u} \mathop {\sum} \limits_{i_2=1}^{n_u} M_{g_{p,j},i_1 i_2} | (\bar u_{k+1,i_1}^* - u_{k,i_1})(\bar u_{k+1,i_2}^* - u_{k,i_2}) |,
\end{array}
\end{equation}

\noindent with $M_{g_j,i_1 i_2}$ and $M_{g_{p,j},i_1 i_2}$ defined in the same manner as $M_{\phi,i_1 i_2}$ in (\ref{eq:lipcon2}):

\begin{equation}\label{eq:lipcongknown2}
-M_{g_j,i_1 i_2} < \frac{\partial^2 g_j}{\partial u_{i_2} \partial u_{i_1} } \Big |_{\bf u} < M_{g_j,i_1 i_2}, \;\; \forall {\bf u} \in \mathcal{I},
\end{equation}

\begin{equation}\label{eq:lipcongp2}
-M_{g_{p,j},i_1 i_2} < \frac{\partial^2 g_{p,j}}{\partial u_{i_2} \partial u_{i_1} } \Big |_{\bf u} < M_{g_{p,j},i_1 i_2}, \;\; \forall {\bf u} \in \mathcal{I}.
\end{equation}

\end{definition}

We now proceed to derive the lower bound on $K_k$ by considering the limitations of (\ref{eq:SCFO1i}), (\ref{eq:SCFO2i}), and (\ref{eq:SCFO7i}) individually.

\begin{theorem}[Sufficiently low filter gain value with respect to Condition (\ref{eq:SCFO7i})]
\label{thm:Kbound1}
Let the project-and-filter approach of (\ref{eq:proj}) and (\ref{eq:inputfilter}) be applied with the assumption that the projection has been feasible for all experimental iterations $0,...,k$. It follows that

\begin{equation}\label{eq:lowerKboundA}
K_k \in \left[  0, 2 \frac{\delta_\phi}{\overline Q_\phi} \right] \Rightarrow (\ref{eq:SCFO7i}).
\end{equation}

\end{theorem}
\begin{proof} Rearranging the inequality (\ref{eq:SCFO7i}), we obtain

\begin{equation}\label{eq:SCFO7iRE}
K_k \leq -2 \frac{\nabla \phi_p ({\bf u}_k)^T (\bar {\bf u}_{k+1}^* - {\bf u}_k)}{\mathop {\sum} \limits_{i_1=1}^{n_u} \mathop {\sum} \limits_{i_2=1}^{n_u} M_{\phi,i_1 i_2} | (\bar u_{k+1,i_1}^* - u_{k,i_1})(\bar u_{k+1,i_2}^* - u_{k,i_2}) |}.
\end{equation}

In finding the minimum value that this bound may take, note that the numerator must be strictly negative due to the feasibility of the projection, and that the expression on the right-hand side may thus be minimized by minimizing $-\nabla \phi_p ({\bf u}_k)^T (\bar {\bf u}_{k+1}^* - {\bf u}_k)$ and maximizing $\mathop {\sum} \limits_{i_1=1}^{n_u} \mathop {\sum} \limits_{i_2=1}^{n_u} M_{\phi,i_1 i_2} | (\bar u_{k+1,i_1}^* - u_{k,i_1})(\bar u_{k+1,i_2}^* - u_{k,i_2}) |$. For the former, one has the lower bound of $\delta_\phi$, which is guaranteed by the projection, while a sufficient upper bound on the latter is given by (\ref{eq:Mupper1}). It thus follows that

\begin{equation}\label{eq:Klowerphi}
K_k \leq 2 \frac{\delta_\phi}{\overline Q_\phi} \Rightarrow (\ref{eq:SCFO7i}),
\end{equation}

\noindent which implies the desired result. \end{proof}

\begin{theorem}[Sufficiently low filter gain value with respect to Condition (\ref{eq:SCFO2i})]
\label{thm:Kbound2}
Let the project-and-filter approach of (\ref{eq:proj}) and (\ref{eq:inputfilter}) be applied with the assumption that the projection has been feasible for all experimental iterations $0,...,k$. It follows that

\begin{equation}\label{eq:lowerKboundB}
K_k \in \left[  0, \mathop {\min} \limits_{j=1,...,n_g} \mathop {\min} \left[ \frac{\epsilon_j}{\overline L_j}, 2\frac{\delta_{g,j}}{\overline Q_j} \right] \right] \Rightarrow (\ref{eq:SCFO2i}).
\end{equation}

\end{theorem}
\begin{proof} The expression (\ref{eq:SCFO2i}) cannot be inverted to obtain an inequality for $K_k$ in the general case. One may, however, use the Lipschitz bounds

\begin{equation}\label{eq:boundgknownlin}
g_{j} ({\bf u}_{k+1}) \leq g_{j} ({\bf u}_k) + \displaystyle \sum_{i=1}^{n_u} \kappa_{ji} | u_{k+1,i} - u_{k,i} |,
\end{equation}

\begin{equation}\label{eq:boundgknown}
\begin{array}{l}
g_j ({\bf u}_{k+1}) \leq g_j ({\bf u}_{k}) + \nabla g_j ({\bf u}_{k})^T ({\bf u}_{k+1} - {\bf u}_{k}) + \\
\hspace{20mm}\displaystyle \frac{1}{2} \sum_{i_1=1}^{n_u} \sum_{i_2=1}^{n_u} M_{g_j,i_1 i_2} | (u_{k+1,i_1} - u_{k,i_1})(u_{k+1,i_2} - u_{k,i_2}) |,
\end{array}
\end{equation}

\noindent and substitute in the filter expression (\ref{eq:inputfilter}) to obtain

\begin{equation}\label{eq:boundgknownlin2}
g_{j} ({\bf u}_{k+1}) \leq g_{j} ({\bf u}_k) + K_k \displaystyle \sum_{i=1}^{n_u} \kappa_{ji} | \bar u_{k+1,i}^* - u_{k,i} |,
\end{equation}

\begin{equation}\label{eq:boundgknown2}
\begin{array}{l}
g_j ({\bf u}_{k+1}) \leq g_j ({\bf u}_{k}) + K_k \nabla g_j ({\bf u}_{k})^T (\bar {\bf u}_{k+1}^* - {\bf u}_{k}) + \\
\hspace{20mm}\displaystyle \frac{K_k^2}{2} \sum_{i_1=1}^{n_u} \sum_{i_2=1}^{n_u} M_{g_j,i_1 i_2} | (\bar u_{k+1,i_1}^* - u_{k,i_1})(\bar u_{k+1,i_2}^* - u_{k,i_2}) |.
\end{array}
\end{equation}

Consider first the case where $g_j ({\bf u}_k) < -\epsilon_j$. Forcing the right-hand side of (\ref{eq:boundgknownlin2}) to be non-positive and rearranging provides the following sufficient condition to ensure that $g_j ({\bf u}_{k+1}) < 0$ for this case:

\begin{equation}\label{eq:Kboundnum1}
K_k \leq \frac{-g_j ({\bf u}_k)}{\displaystyle \sum_{i=1}^{n_u} \kappa_{ji} | \bar u_{k+1,i}^* - u_{k,i} |}.
\end{equation}

\noindent This bound may be made global by minimizing the numerator (bounded from below by $\epsilon_j$) and maximizing the denominator (bounded from above by $\overline L_{j}$), thus leading to the following implication:

\begin{equation}\label{eq:case1num}
g_j ({\bf u}_k) < -\epsilon_j \wedge K_k \leq \frac{\epsilon_j}{\overline L_j} \Rightarrow g_j ({\bf u}_{k+1}) < 0.
\end{equation}

Considering the alternative where $g_j ({\bf u}_k) \geq -\epsilon_j$, we may employ (\ref{eq:boundgknown2}), first noting that, since $g_j ({\bf u}_k) \leq 0$, the following must hold as well:

\begin{equation}\label{eq:boundgknown3}
\begin{array}{l}
g_j ({\bf u}_{k+1}) \leq K_k \nabla g_j ({\bf u}_{k})^T (\bar {\bf u}_{k+1}^* - {\bf u}_{k}) + \\
\hspace{20mm}\displaystyle \frac{K_k^2}{2} \sum_{i_1=1}^{n_u} \sum_{i_2=1}^{n_u} M_{g_j,i_1 i_2} | (\bar u_{k+1,i_1}^* - u_{k,i_1})(\bar u_{k+1,i_2}^* - u_{k,i_2}) |.
\end{array}
\end{equation}

Recalling that $K_k$ is restricted to be non-negative, we note that the particular case of $K_k = 0$ trivially implies $g_j ({\bf u}_{k+1}) \leq 0$. Suppose then that $K_k > 0$ and set the right-hand side to be non-positive. Dividing by $K_k$ then provides the sufficient condition to ensure that $g_j ({\bf u}_{k+1}) \leq 0$:

\begin{equation}\label{eq:boundgknown4}
\begin{array}{l}
\nabla g_j ({\bf u}_{k})^T (\bar {\bf u}_{k+1}^* - {\bf u}_{k}) + \vspace{1mm}\\
\hspace{20mm}\displaystyle \frac{K_k}{2} \sum_{i_1=1}^{n_u} \sum_{i_2=1}^{n_u} M_{g_j,i_1 i_2} | (\bar u_{k+1,i_1}^* - u_{k,i_1})(\bar u_{k+1,i_2}^* - u_{k,i_2}) | \leq 0,
\end{array}
\end{equation}

\noindent which may be rearranged to yield

\begin{equation}\label{eq:SCFO2iRE}
K_k \leq -2 \frac{\nabla g_j ({\bf u}_k)^T (\bar {\bf u}_{k+1}^* - {\bf u}_k)}{\mathop {\sum} \limits_{i_1=1}^{n_u} \mathop {\sum} \limits_{i_2=1}^{n_u} M_{g_j,i_1 i_2} | (\bar u_{k+1,i_1}^* - u_{k,i_1})(\bar u_{k+1,i_2}^* - u_{k,i_2}) |}.
\end{equation}

Using, again, the feasibility of the projection, which bounds the numerator, and the bound (\ref{eq:Mupper2}), which bounds the denominator, we obtain the global implication

\begin{equation}\label{eq:case2num}
g_j ({\bf u}_k) \geq -\epsilon_j \wedge K_k \leq 2 \frac{\delta_{g,j}}{\overline Q_j} \Rightarrow g_j ({\bf u}_{k+1}) \leq 0.
\end{equation}

Taking both cases into consideration and taking the minimum over $j = 1,...,n_g$ to account for all of the $n_g$ constraints then leads to the desired result. \end{proof}

\begin{theorem}[Sufficiently low filter gain value with respect to Condition (\ref{eq:SCFO1i})]
\label{thm:Kbound3}
Let the project-and-filter approach of (\ref{eq:proj}) and (\ref{eq:inputfilter}) be applied with the assumption that the projection has been feasible for all experimental iterations $0,...,k$. It follows that

\begin{equation}\label{eq:lowerKboundC}
K_k \in \left[  0, \mathop {\min} \limits_{j=1,...,n_{g_p}} \frac{\mathop {\min} \left[ (1-\gamma_j)\epsilon_{p,j},  \displaystyle  2 (1-\gamma_j) \frac{ \delta_{g_p,j}^2}{\overline Q_{p,j}}, -g_{p,j} ({\bf u}_0) \right]}{\overline L_{p,j}} \right] \Rightarrow (\ref{eq:SCFO1i}).
\end{equation}

\end{theorem}
\begin{proof} Rearranging (\ref{eq:SCFO1i}) leads to

\begin{equation}\label{eq:SCFO1iRE}
K_k \leq \frac{-g_{p,j}({\bf u}_k)}{\mathop {\sum} \limits_{i=1}^{n_u} \kappa_{p,ji}|\bar u_{k+1,i}^* - u_{k,i}|}.
\end{equation}

As in Theorem \ref{thm:Kbound2}, the right-hand side may be lower bounded by maximizing the denominator and minimizing the numerator, with the upper bound on the former given by (\ref{eq:Kupper}). However, the trivial lower bound on the numerator -- i.e., 0, which follows from the guarantee of feasibility -- is insufficient for our purposes and so a better way to bound the value of $-g_{p,j}({\bf u}_k)$ is needed. Two ways of sequentially bounding $-g_{p,j}({\bf u}_k)$ with respect to $-g_{p,j}({\bf u}_{k-1})$ are now considered.

Consider first the nonstrict version of the bound in (\ref{eq:bound1U}):

\begin{equation}\label{eq:bound1Uns}
g_{p,j} ({\bf u}_{k+1}) \leq g_{p,j} ({\bf u}_k) + \displaystyle \sum_{i=1}^{n_u} \tilde \kappa_{p,ji} | u_{k+1,i} - u_{k,i} |,
\end{equation}

\noindent which may hold with equality even when ${\bf u}_{k+1} \neq {\bf u}_{k}$ \citep{Bunin:Lip}. From Definition \ref{def:strictness}, one has that $\tilde \kappa_{p,ji} \leq \gamma_j \kappa_{p,ji}$, which then allows:

\begin{equation}\label{eq:bound1Uns2}
\begin{array}{l}
\displaystyle g_{p,j} ({\bf u}_k) + \displaystyle \sum_{i=1}^{n_u} \tilde \kappa_{p,ji} | u_{k+1,i} - u_{k,i} | \leq g_{p,j} ({\bf u}_k) + \gamma_j \displaystyle \sum_{i=1}^{n_u} \kappa_{p,ji} | u_{k+1,i} - u_{k,i} |  \\
\displaystyle \Rightarrow g_{p,j} ({\bf u}_{k+1}) \leq g_{p,j} ({\bf u}_k) + \gamma_j \displaystyle \sum_{i=1}^{n_u} \kappa_{p,ji} | u_{k+1,i} - u_{k,i} |.
\end{array}
\end{equation}

At the same time, since $K_k$ is always chosen so that Condition (\ref{eq:SCFO1}) is satisfied, this condition may be exploited and rearranged to yield

\begin{equation}\label{eq:SCFO1RE}
\displaystyle \mathop {\sum} \limits_{i=1}^{n_u} \kappa_{p,ji} | u_{k+1,i} - u_{k,i} | \leq -g_{p,j}({\bf u}_k),
\end{equation}

\noindent which, following the multiplication of both sides of (\ref{eq:SCFO1RE}) by $\gamma_j$, allows for the bound in (\ref{eq:bound1Uns2}) to be developed further:

\begin{equation}\label{eq:bound1Uns3}
\begin{array}{l}
\gamma_j \displaystyle \sum_{i=1}^{n_u} \kappa_{p,ji} | u_{k+1,i} - u_{k,i} | \leq -\gamma_j g_{p,j} ({\bf u}_k) \\
\Rightarrow g_{p,j} ({\bf u}_{k+1}) \leq g_{p,j} ({\bf u}_k) - \gamma_j g_{p,j}({\bf u}_k) = (1-\gamma_j)g_{p,j}({\bf u}_k) \\
\Rightarrow g_{p,j} ({\bf u}_{k+1}) \leq (1-\gamma_j)g_{p,j}({\bf u}_k) \\
\Rightarrow g_{p,j} ({\bf u}_{k}) \leq (1-\gamma_j)g_{p,j}({\bf u}_{k-1}) \\
\Leftrightarrow -g_{p,j} ({\bf u}_{k}) \geq (1-\gamma_j)(-g_{p,j}({\bf u}_{k-1})).
\end{array}
\end{equation}

\noindent Here, we have shifted the indices back to bound $g_{p,j}({\bf u}_k)$ with respect to $g_{p,j}({\bf u}_{k-1})$ -- the bound above being valid for \emph{any} two consecutive iterations.

The second way of bounding $-g_{p,j} ({\bf u}_{k})$ with respect to $-g_{p,j} ({\bf u}_{k-1})$ considers the specific case where $-g_{p,j} ({\bf u}_{k-1}) \leq \epsilon_{p,j}$, for which one may exploit, from the projection, the guarantee that $\nabla g_{p,j} ({\bf u}_{k-1})^T (\bar {\bf u}_{k}^* - {\bf u}_{k-1}) \leq -\delta_{g_p,j}$. This is done by employing the quadratic upper bound for $g_{p,j}$ and taking the steps analogous to (\ref{eq:boundgknown}) and (\ref{eq:boundgknown2}), which leads to

\begin{equation}\label{eq:boundgp}
\begin{array}{l}
g_{p,j} ({\bf u}_{k}) - g_{p,j} ({\bf u}_{k-1}) \leq K_{k-1} \nabla g_{p,j} ({\bf u}_{k-1})^T (\bar {\bf u}_{k}^* - {\bf u}_{k-1}) + \\
\hspace{10mm}\displaystyle \frac{K_{k-1}^2}{2} \sum_{i_1=1}^{n_u} \sum_{i_2=1}^{n_u} M_{g_{p,j},i_1 i_2} | (\bar u_{k,i_1}^* - u_{k-1,i_1})(\bar u_{k,i_2}^* - u_{k-1,i_2}) |.
\end{array}
\end{equation}

By forcing the right-hand side of (\ref{eq:boundgp}) to be non-positive, one sees, by the same analysis as in Theorem \ref{thm:Kbound2}, that

\begin{equation}\label{eq:Kimply1}
\displaystyle K_{k-1} \leq 2 \frac{\delta_{g_p,j}}{\overline Q_{p,j}} \Rightarrow - g_{p,j} ({\bf u}_{k-1}) \leq -g_{p,j} ({\bf u}_{k}).
\end{equation}

Because Condition (\ref{eq:SCFO1i}) must be fulfilled by the experiment at $k-1$, it follows that

\begin{equation}\label{eq:SCFO1iREshift}
K_{k-1} \leq \frac{-g_{p,j}({\bf u}_{k-1})}{\mathop {\sum} \limits_{i=1}^{n_u} \kappa_{p,ji}|\bar u_{k,i}^* - u_{k-1,i}|},
\end{equation}

\noindent which may be used to extend the implication of (\ref{eq:Kimply1}):

\begin{equation}\label{eq:implyextend1}
\displaystyle \frac{-g_{p,j}({\bf u}_{k-1})}{\mathop {\sum} \limits_{i=1}^{n_u} \kappa_{p,ji}|\bar u_{k,i}^* - u_{k-1,i}|} \leq 2 \frac{\delta_{g_p,j}}{\overline Q_{p,j}} \Rightarrow K_{k-1} \leq 2 \frac{\delta_{g_p,j}}{\overline Q_{p,j}}.
\end{equation}

This now allows us the statement:

\begin{equation}\label{eq:Kimply2}
\displaystyle \frac{-g_{p,j}({\bf u}_{k-1})}{\mathop {\sum} \limits_{i=1}^{n_u} \kappa_{p,ji}|\bar u_{k,i}^* - u_{k-1,i}|} \leq 2 \frac{\delta_{g_p,j}}{\overline Q_{p,j}} \Rightarrow - g_{p,j} ({\bf u}_{k-1}) \leq -g_{p,j} ({\bf u}_{k}).
\end{equation}

This may be advanced further by upper bounding the left-hand side by deriving a lower bound on the denominator. Note that, from the projection:

\begin{equation}\label{eq:suff1re}
\begin{array}{r}
\nabla g_{p,j}({\bf{u}}_{k-1})^T (\bar {\bf{u}}_{k}^* - {\bf{u}}_{k-1}) = \displaystyle \sum_{i=1}^{n_u} \frac{\partial g_{p,j}}{\partial u_i} \Big |_{{\bf u}_k} \left( \bar u_{k,i}^* - u_{k-1,i} \right) \leq -\delta_{g_p,j} \\
\Leftrightarrow \displaystyle \sum_{i=1}^{n_u} \frac{\partial g_{p,j}}{\partial u_i} \Big |_{{\bf u}_k} \left( u_{k-1,i} - \bar u_{k,i}^* \right) \geq \delta_{g_p,j}.
\end{array}
\end{equation}

Since $xy \leq |x||y|$ for any $x,y \in \mathbb{R}$ and $|x||y| < \overline x |y|$ for $|x| < \overline x$ and $y \neq 0$, it is readily seen that

\begin{equation}\label{eq:suff1re2}
\begin{array}{r}
\displaystyle \sum_{i=1}^{n_u} \kappa_{p,ji} | \bar u_{k,i}^* - u_{k-1,i} | > \displaystyle \sum_{i=1}^{n_u} \frac{\partial g_{p,j}}{\partial u_i} \Big |_{{\bf u}_k} \left( u_{k-1,i} - \bar u_{k,i}^* \right) \geq \delta_{g_p,j} \\
\Rightarrow \displaystyle \frac{-g_{p,j}({\bf u}_{k-1})}{\mathop {\sum} \limits_{i=1}^{n_u} \kappa_{p,ji}|\bar u_{k,i}^* - u_{k-1,i}|} < \displaystyle \frac{-g_{p,j}({\bf u}_{k-1})}{\delta_{g_p,j}},
\end{array}
\end{equation}

\noindent which, by

\begin{equation}\label{eq:implyextend2}
\displaystyle \displaystyle \frac{-g_{p,j}({\bf u}_{k-1})}{\delta_{g_p,j}} \leq 2 \frac{\delta_{g_p,j}}{\overline Q_{p,j}} \\
\Rightarrow \displaystyle \displaystyle \frac{-g_{p,j}({\bf u}_{k-1})}{\mathop {\sum} \limits_{i=1}^{n_u} \kappa_{p,ji}|\bar u_{k,i}^* - u_{k-1,i}|} \leq 2 \frac{\delta_{g_p,j}}{\overline Q_{p,j}},
\end{equation}

\noindent finally allows

\begin{equation}\label{eq:Kimply3}
\displaystyle \frac{-g_{p,j}({\bf u}_{k-1})}{\delta_{g_p,j}} \leq 2 \frac{\delta_{g_p,j}}{\overline Q_{p,j}} \Rightarrow - g_{p,j} ({\bf u}_{k-1}) \leq -g_{p,j} ({\bf u}_{k}),
\end{equation}

\noindent or

\begin{equation}\label{eq:Kimply4}
\displaystyle -g_{p,j}({\bf u}_{k-1}) \leq 2 \frac{\delta_{g_p,j}^2}{\overline Q_{p,j}} \Rightarrow - g_{p,j} ({\bf u}_{k-1}) \leq -g_{p,j} ({\bf u}_{k}),
\end{equation}

\noindent provided, again, that $-g_{p,j} ({\bf u}_{k-1}) \leq \epsilon_{p,j}$.

Using (\ref{eq:bound1Uns3}) and (\ref{eq:Kimply4}), one may now proceed to derive useful bounds on $-g_{p,j} ({\bf u}_k)$, which will be conditional in nature since the relationships between $-g_{p,j} ({\bf u}_{k-1})$, $\epsilon_{p,j}$, and $2 \delta_{g_p,j}^2/\overline Q_{p,j}$ will influence whether (\ref{eq:bound1Uns3}) and/or (\ref{eq:Kimply4}) may be used and whether or not they are useful. The procedure taken here is to split up all of the possibilities by considering the following questions:  
\newpage
\begin{enumerate}[A.]
\setlength\itemindent{.65cm} \item $-g_{p,j} ({\bf u}_{k-1}) > \epsilon_{p,j}$? (True/False) \vspace{1mm}
\setlength\itemindent{.65cm} \item $\displaystyle -g_{p,j} ({\bf u}_{k-1}) > 2 \frac{\delta_{g_p,j}^2}{\overline Q_{p,j}}$? (T/F)
\setlength\itemindent{.65cm} \item $\displaystyle \epsilon_{p,j} > 2 \frac{\delta_{g_p,j}^2}{\overline Q_{p,j}}$? (T/F)
\end{enumerate}

\noindent and generating all of the possible scenarios based on the eight permutations of the answers:

$$
\begin{array}{lllll}
\begin{array}{c}{\rm Scenario\;1} \\ {\rm (TTT)} \end{array} & : & \begin{array}{l} ({\rm A}): -g_{p,j} ({\bf u}_{k-1}) > \epsilon_{p,j} \vspace{1mm} \\ ({\rm B}): \displaystyle -g_{p,j} ({\bf u}_{k-1}) > 2 \frac{\delta_{g_p,j}^2}{\overline Q_{p,j}} \\  ({\rm C}): \displaystyle \epsilon_{p,j} > 2 \frac{\delta_{g_p,j}^2}{\overline Q_{p,j}} \end{array} & \Leftrightarrow & -g_{p,j} ({\bf u}_{k-1}) > \displaystyle \epsilon_{p,j} > 2 \frac{\delta_{g_p,j}^2}{\overline Q_{p,j}} \vspace{3mm} \\

\begin{array}{c}{\rm Scenario\;2} \\ {\rm (TTF)} \end{array} & : & \begin{array}{l} ({\rm A}): -g_{p,j} ({\bf u}_{k-1}) > \epsilon_{p,j} \vspace{1mm} \\ ({\rm B}): \displaystyle -g_{p,j} ({\bf u}_{k-1}) > 2 \frac{\delta_{g_p,j}^2}{\overline Q_{p,j}} \\  ({\rm C}): \displaystyle \epsilon_{p,j} \leq 2 \frac{\delta_{g_p,j}^2}{\overline Q_{p,j}} \end{array} & \Leftrightarrow & \displaystyle -g_{p,j} ({\bf u}_{k-1}) > 2 \frac{\delta_{g_p,j}^2}{\overline Q_{p,j}} \geq  \epsilon_{p,j} \vspace{3mm} \\

\begin{array}{c}{\rm Scenario\;3} \\ {\rm (TFT)} \end{array} & : & \begin{array}{l} ({\rm A}): -g_{p,j} ({\bf u}_{k-1}) > \epsilon_{p,j} \vspace{1mm} \\ ({\rm B}): \displaystyle -g_{p,j} ({\bf u}_{k-1}) \leq 2 \frac{\delta_{g_p,j}^2}{\overline Q_{p,j}} \\  ({\rm C}): \displaystyle \epsilon_{p,j} > 2 \frac{\delta_{g_p,j}^2}{\overline Q_{p,j}} \end{array} & \Rightarrow & {\rm impossible}\; ({\rm A} \wedge {\rm C} \rightarrow \lnot {\rm B}) \vspace{3mm} \\

\begin{array}{c}{\rm Scenario\;4} \\ {\rm (TFF)} \end{array} & : & \begin{array}{l} ({\rm A}): -g_{p,j} ({\bf u}_{k-1}) > \epsilon_{p,j} \vspace{1mm} \\ ({\rm B}): \displaystyle -g_{p,j} ({\bf u}_{k-1}) \leq 2 \frac{\delta_{g_p,j}^2}{\overline Q_{p,j}} \\  ({\rm C}): \displaystyle \epsilon_{p,j} \leq 2 \frac{\delta_{g_p,j}^2}{\overline Q_{p,j}} \end{array} & \Rightarrow & \displaystyle 2 \frac{\delta_{g_p,j}^2}{\overline Q_{p,j}} \geq -g_{p,j} ({\bf u}_{k-1}) > \epsilon_{p,j} \vspace{3mm} \\

\begin{array}{c}{\rm Scenario\;5} \\ {\rm (FTT)} \end{array} & : & \begin{array}{l} ({\rm A}): -g_{p,j} ({\bf u}_{k-1}) \leq \epsilon_{p,j} \vspace{1mm} \\ ({\rm B}): \displaystyle -g_{p,j} ({\bf u}_{k-1}) > 2 \frac{\delta_{g_p,j}^2}{\overline Q_{p,j}} \\  ({\rm C}): \displaystyle \epsilon_{p,j} > 2 \frac{\delta_{g_p,j}^2}{\overline Q_{p,j}} \end{array} & \Leftrightarrow & \displaystyle \epsilon_{p,j} \geq -g_{p,j} ({\bf u}_{k-1}) > 2 \frac{\delta_{g_p,j}^2}{\overline Q_{p,j}}  \vspace{3mm} \\

\begin{array}{c}{\rm Scenario\;6} \\ {\rm (FTF)} \end{array} & : & \begin{array}{l} ({\rm A}): -g_{p,j} ({\bf u}_{k-1}) \leq \epsilon_{p,j} \vspace{1mm} \\ ({\rm B}): \displaystyle -g_{p,j} ({\bf u}_{k-1}) > 2 \frac{\delta_{g_p,j}^2}{\overline Q_{p,j}} \\  ({\rm C}): \displaystyle \epsilon_{p,j} \leq 2 \frac{\delta_{g_p,j}^2}{\overline Q_{p,j}} \end{array} & \Rightarrow & {\rm impossible}\; ({\rm A} \wedge {\rm B} \rightarrow \lnot {\rm C})  \vspace{3mm} \\

\end{array}
$$

$$
\begin{array}{lllll}

\begin{array}{c}{\rm Scenario\;7} \\ {\rm (FFT)} \end{array} & : & \begin{array}{l} ({\rm A}): -g_{p,j} ({\bf u}_{k-1}) \leq \epsilon_{p,j} \vspace{1mm} \\ ({\rm B}): \displaystyle -g_{p,j} ({\bf u}_{k-1}) \leq 2 \frac{\delta_{g_p,j}^2}{\overline Q_{p,j}} \\  ({\rm C}): \displaystyle \epsilon_{p,j} > 2 \frac{\delta_{g_p,j}^2}{\overline Q_{p,j}} \end{array} & \Rightarrow & \displaystyle \epsilon_{p,j} > 2 \frac{\delta_{g_p,j}^2}{\overline Q_{p,j}} \geq   -g_{p,j} ({\bf u}_{k-1}) \vspace{3mm} \\

\begin{array}{c}{\rm Scenario\;8} \\ {\rm (FFF)} \end{array} & : & \begin{array}{l} ({\rm A}): -g_{p,j} ({\bf u}_{k-1}) \leq \epsilon_{p,j} \vspace{1mm} \\ ({\rm B}): \displaystyle -g_{p,j} ({\bf u}_{k-1}) \leq 2 \frac{\delta_{g_p,j}^2}{\overline Q_{p,j}} \\  ({\rm C}): \displaystyle \epsilon_{p,j} \leq 2 \frac{\delta_{g_p,j}^2}{\overline Q_{p,j}} \end{array} & \Leftrightarrow & \displaystyle 2 \frac{\delta_{g_p,j}^2}{\overline Q_{p,j}} \geq \epsilon_{p,j} \geq    -g_{p,j} ({\bf u}_{k-1})

\end{array}
$$

We now go through all of the scenarios (excluding 3 and 6, which cannot occur), derive lower bounds on $-g_{p,j} ({\bf u}_k)$ for each, and take their minimum to obtain an overall bound that holds for all of the different possibilities. Treating Scenarios 1, 2, and 4 first, we note that we can simply apply (\ref{eq:bound1Uns3}) to obtain:

\begin{equation}\label{eq:scen1}
{\rm Scenario}\; 1/2/4 \Rightarrow -g_{p,j} ({\bf u}_k) > (1-\gamma_j) \epsilon_{p,j},
\end{equation}

\noindent which follows from the fact that $-g_{p,j} ({\bf u}_{k-1}) > \epsilon_{p,j}$ in these scenarios.

In a similar manner, exploiting $-g_{p,j} ({\bf u}_{k-1}) > 2 \delta_{g_p,j}^2/\overline Q_{p,j}$ in Scenario 5 together with (\ref{eq:bound1Uns3}) allows:

\begin{equation}\label{eq:scen2}
{\rm Scenario}\; 5 \Rightarrow -g_{p,j} ({\bf u}_k) > 2 (1-\gamma_j) \frac{\delta_{g_p,j}^2}{\overline Q_{p,j}}.
\end{equation}

Consider now Scenarios 7 and 8, which are more involved since one cannot simply employ (\ref{eq:bound1Uns3}) to any useful end, the value of $-g_{p,j} ({\bf u}_{k-1})$ not being explicitly lower bounded. However, exploiting the fact that $-g_{p,j} ({\bf u}_{k-1}) \leq 2 \delta_{g_p,j}^2/\overline Q_{p,j}$ together with (\ref{eq:Kimply4}) allows:

\begin{equation}\label{eq:scen3}
{\rm Scenario}\; 7/8 \Rightarrow -g_{p,j} ({\bf u}_k) \geq -g_{p,j} ({\bf u}_{k-1}).
\end{equation}

\noindent To advance this further, one now needs to bound $-g_{p,j} ({\bf u}_{k-1})$. Again, consider the six possible scenarios but this time for $-g_{p,j} ({\bf u}_{k-2})$. Clearly, Scenarios 1/2/4/5 for $-g_{p,j} ({\bf u}_{k-2})$ coupled with Scenarios 7/8 for $-g_{p,j} ({\bf u}_{k-1})$ will yield the same bounds as (\ref{eq:scen1}) and (\ref{eq:scen2}). In fact, one can easily see that having at least one occurrence of Scenarios 1/2/4/5 for some experimental iteration between 0 and $k-1$ ensures the validity of  (\ref{eq:scen1}) and (\ref{eq:scen2}) for $k$, as shifting to Scenarios 7/8 cannot push these lower bounds any further. The only remaining case of interest is when Scenarios 7/8 occur for all $0,...,k-1$. If this is so, then the value of $-g_{p,j}$ can never go below its initial value, i.e.:

\begin{equation}\label{eq:scen4}
{\rm Scenario}\; 7/8 \; {\rm for\;all} \; 0,...,k-1 \Rightarrow -g_{p,j} ({\bf u}_k) \geq -g_{p,j} ({\bf u}_{0}).
\end{equation}

As such, the minimum value that $-g_{p,j} ({\bf u}_k)$ can ever achieve for any scenario is bounded from below by

\begin{equation}\label{eq:gpmin}
-g_{p,j} ({\bf u}_k) \geq \mathop {\min} \left[ (1-\gamma_j) \epsilon_{p,j}, 2 (1-\gamma_j) \frac{\delta_{g_p,j}^2}{\overline Q_{p,j}}, -g_{p,j} ({\bf u}_{0})  \right].
\end{equation}

\noindent Using this as the global lower bound on the numerator of (\ref{eq:SCFO1iRE}), upper bounding the denominator by $\overline L_{p,j}$, and taking the minimum over the constraints $j = 1,...,n_{g_p}$ then yields the desired result. \end{proof}

\begin{corollary}[Lower bound on filter gain value]
\label{cor:Kbound}
Let the project-and-filter approach of (\ref{eq:proj}) and (\ref{eq:inputfilter}) be applied with the assumption that the projection has been feasible for all experimental iterations $0,...,k$ and that $K_k$ is chosen as the maximum value on the interval $[0,1]$ subject to the limitations of (\ref{eq:SCFO1i}), (\ref{eq:SCFO2i}), and (\ref{eq:SCFO7i}). The value

\begin{equation}\label{eq:lowerKbound}
 \underline K = \mathop {\min} \left[ \begin{array}{c} \displaystyle 2 \frac{\delta_\phi}{\overline Q_\phi}, \vspace{1mm} \\ \displaystyle \mathop {\min} \limits_{j=1,...,n_g} \mathop {\min} \left[ \frac{\epsilon_j}{\overline L_j}, 2\frac{\delta_{g,j}}{\overline Q_j} \right], \vspace{1mm} \\
\displaystyle  \mathop {\min} \limits_{j=1,...,n_{g_p}} \frac{\mathop {\min} \left[ (1-\gamma_j)\epsilon_{p,j},  \displaystyle  2 (1-\gamma_j) \frac{ \delta_{g_p,j}^2}{\overline Q_{p,j}}, -g_{p,j} ({\bf u}_0) \right]}{\overline L_{p,j}} 
\end{array}\right]
\end{equation}

\noindent is a valid lower bound on the filter gain, with $0 < \underline K \leq K_k$.

\end{corollary}
\begin{proof} The result follows from Theorems \ref{thm:Kbound1}-\ref{thm:Kbound3}, as $\underline K$ is guaranteed to satisfy (\ref{eq:SCFO1i}), (\ref{eq:SCFO2i}), and (\ref{eq:SCFO7i}) and will thus be found by the line search. \end{proof}

\subsection{Convergence to a Fixed Point}
We now address what happens when the projection is \emph{not} feasible at all experimental iterations. As already mentioned, we fix ${\bf u}_{k+1} := {\bf u}_k$ whenever this infeasibility is encountered, which effectively ensures that the scheme converge once Problem (\ref{eq:proj}) does not admit a solution. It will now be proven that this must occur after a finite number of experiments that can be upper bounded by a function of the problem characteristics and the projection parameters. Another strictness definition is required first, however.

\begin{definition}[The strictness coefficient for the Lipschitz constants of the derivatives of the cost function]
\label{def:strictness2}
Denote by $\tilde M_{\phi,i_1 i_2}$ the \emph{nonstrict} Lipschitz constants for the derivatives of the cost function:

\begin{equation}\label{eq:lipcon2ns}
-\tilde M_{\phi,i_1 i_2} \leq \frac{\partial^2 \phi_p}{\partial u_{i_2} \partial u_{i_1} } \Big |_{\bf u} \leq \tilde M_{\phi,i_1 i_2}, \;\; \forall {\bf u} \in \mathcal{I}.
\end{equation}

The strictness coefficient for the Lipschitz constants of the cost derivatives is then defined as

\begin{equation}\label{eq:gammaM}
\gamma_\phi = \mathop {\max} \limits_{\footnotesize{\begin{array}{c}i_1 = 1,...,n_u \\ i_2 = 1,...,n_u \end{array}}} \frac{\tilde M_{\phi,i_1 i_2}}{M_{\phi,i_1 i_2}} < 1.
\end{equation}

\end{definition}

\begin{theorem}[An upper bound on the number of experiments prior to convergence to a fixed point]
\label{thm:iterations}
Let the project-and-filter approach of (\ref{eq:proj}) and (\ref{eq:inputfilter}) be applied at every experimental iteration where the projection is feasible, with $K_k$ chosen as the maximum value on the interval $[0,1]$ subject to the limitations of (\ref{eq:SCFO1i}), (\ref{eq:SCFO2i}), and (\ref{eq:SCFO7i}), and let ${\bf u}_{k+1} := {\bf u}_k$ if the projection is not feasible. It follows that the number of experiments for which the projection can be feasible cannot exceed

\begin{equation}\label{eq:maxiter}
\frac{\underline \phi_p - \phi_p({\bf u}_0)}{\mathop {\max} \left[ \displaystyle \underline K \left( \underline K  \frac{\gamma_\phi \overline Q_\phi}{2}  - \delta_\phi \right), 2(\gamma_\phi - 1) \frac{ \delta_\phi^2}{\overline Q_\phi} \right]},
\end{equation}

\noindent where $\underline \phi_p$ is the global minimum cost function value for Problem (\ref{eq:mainprob}).

\end{theorem}
\begin{proof}
We start by deriving the minimal cost decrease that must occur whenever the projection is feasible. Taking the bound in (\ref{eq:bound2U}), consider its nonstrict version:

\begin{equation}\label{eq:costbound1}
\begin{array}{l}
\phi_p({\bf u}_{k+1}) - \phi_p({\bf u}_{k}) \leq \nabla \phi_p({\bf u}_{k})^T ({\bf u}_{k+1} - {\bf u}_{k}) +\\
\hspace{20mm} \displaystyle \frac{1}{2} \mathop {\sum} \limits_{i_1=1}^{n_u} \mathop {\sum} \limits_{i_2=1}^{n_u} \tilde M_{\phi,i_1 i_2} | ( u_{k+1,i_1}- u_{k,i_1})(u_{k+1,i_2} - u_{k,i_2}) |,
\end{array}
\end{equation}

\noindent which may hold with equality even if ${\bf u}_{k+1} \neq {\bf u}_k$, and employ Definition \ref{def:strictness2}, which states that $\tilde M_{\phi,i_1 i_2} \leq \gamma_\phi M_{\phi,i_1 i_2}$. Since

\begin{equation}\label{eq:Mimply0}
\tilde M_{\phi,i_1 i_2} \leq \gamma_\phi M_{\phi,i_1 i_2} \Rightarrow \begin{array}{l} \tilde M_{\phi,i_1 i_2} | ( u_{k+1,i_1}- u_{k,i_1})(u_{k+1,i_2} - u_{k,i_2}) | \vspace{1mm} \\ \hspace{5mm} \leq \gamma_\phi M_{\phi,i_1 i_2} | ( u_{k+1,i_1}- u_{k,i_1})(u_{k+1,i_2} - u_{k,i_2}) |, \end{array}
\end{equation}

\noindent it follows that

\begin{equation}\label{eq:costbound2}
\begin{array}{l}
\phi_p({\bf u}_{k+1}) - \phi_p({\bf u}_{k}) \leq \nabla \phi_p({\bf u}_{k})^T ({\bf u}_{k+1} - {\bf u}_{k}) +\\
\hspace{20mm} \displaystyle \frac{\gamma_\phi}{2} \mathop {\sum} \limits_{i_1=1}^{n_u} \mathop {\sum} \limits_{i_2=1}^{n_u} M_{\phi,i_1 i_2} | ( u_{k+1,i_1}- u_{k,i_1})(u_{k+1,i_2} - u_{k,i_2}) |.
\end{array}
\end{equation}

Next, substitute the input filter law (\ref{eq:inputfilter}) into the right-hand side:

\begin{equation}\label{eq:costbound3}
\begin{array}{l}
\phi_p({\bf u}_{k+1}) - \phi_p({\bf u}_{k}) \leq K_k \nabla \phi_p({\bf u}_{k})^T (\bar {\bf u}_{k+1}^* - {\bf u}_{k}) +\\
\hspace{20mm} K_k^2 \displaystyle \frac{\gamma_\phi}{2} \mathop {\sum} \limits_{i_1=1}^{n_u} \mathop {\sum} \limits_{i_2=1}^{n_u} M_{\phi,i_1 i_2} | ( \bar u_{k+1,i_1}^*- u_{k,i_1})(\bar u_{k+1,i_2}^* - u_{k,i_2}) |
\end{array}.
\end{equation}

\noindent Considering first the case where $\gamma_\phi > 0$, note that this bound is strictly convex (quadratic) in $K_k$, and as such is strictly negative on the interval 

\begin{equation}\label{eq:interval1}
K_k \in \left( 0,-2 \frac{\nabla \phi_p({\bf u}_{k})^T (\bar {\bf u}_{k+1}^* - {\bf u}_{k})}{\gamma_\phi \mathop {\sum} \limits_{i_1=1}^{n_u} \mathop {\sum} \limits_{i_2=1}^{n_u} M_{\phi,i_1 i_2} | ( \bar u_{k+1,i_1}^*- u_{k,i_1})(\bar u_{k+1,i_2}^* - u_{k,i_2}) |} \right),
\end{equation}

\noindent i.e., between its zeros.

Using the already obtained results, it may be shown that $K_k$ will lie in a subinterval of (\ref{eq:interval1}). First, the left side of the interval in which $K_k$ will lie may be taken as $\underline K$, as it has been proven (Corollary \ref{cor:Kbound}) that $K_k$ cannot be any lower than this value. For the right side, one may use the upper bound (\ref{eq:SCFO7iRE}), since this is imposed at every experimental iteration $k$. As such, let us consider the maximum value that the bound (\ref{eq:costbound3}) can take on the interval

$$
K_k \in \left[ \underline K, -2 \frac{\nabla \phi_p ({\bf u}_k)^T (\bar {\bf u}_{k+1}^* - {\bf u}_k)}{\mathop {\sum} \limits_{i_1=1}^{n_u} \mathop {\sum} \limits_{i_2=1}^{n_u} M_{\phi,i_1 i_2} | (\bar u_{k+1,i_1}^* - u_{k,i_1})(\bar u_{k+1,i_2}^* - u_{k,i_2}) |} \right],
$$

\noindent which is easily seen to lie inside the open interval (\ref{eq:interval1}) for which a strict decrease in the cost function value is guaranteed -- this follows from $\underline K > 0$ and $\gamma_\phi < 1$. The immediate consequence is that a strict decrease in the cost function value is guaranteed for all experimental iterations where the projection is feasible.

To calculate the actual worst-case decrease, we use the strict convexity of (\ref{eq:costbound3}) in $K_k$, which implies that the bound must attain its maximum at the interval boundary. This leads to only two possibilities. First, if the maximum is attained on the left boundary, the bound (\ref{eq:costbound3}) becomes

\begin{equation}\label{eq:costbound4}
\begin{array}{l}
\phi_p({\bf u}_{k+1}) - \phi_p({\bf u}_{k}) \leq \underline K \nabla \phi_p({\bf u}_{k})^T (\bar {\bf u}_{k+1}^* - {\bf u}_{k}) +\\
\hspace{20mm} \underline K^2 \displaystyle \frac{\gamma_\phi}{2} \mathop {\sum} \limits_{i_1=1}^{n_u} \mathop {\sum} \limits_{i_2=1}^{n_u} M_{\phi,i_1 i_2} | ( \bar u_{k+1,i_1}^*- u_{k,i_1})(\bar u_{k+1,i_2}^* - u_{k,i_2}) |.
\end{array}
\end{equation}

To make this bound global and independent of $k$, one may use the bounds $\nabla \phi_{p} ({\bf u}_k)^T$ $(\bar {\bf u}_{k+1}^* - {\bf u}_k) \leq -\delta_{\phi}$, which follows from the feasibility of the projection, and (\ref{eq:Mupper1}) to obtain

\begin{equation}\label{eq:costbound5}
\phi_p({\bf u}_{k+1}) - \phi_p({\bf u}_{k}) \leq \underline K \left( \underline K \displaystyle \frac{\gamma_\phi \overline Q_\phi}{2}  - \delta_\phi \right)   < 0.
\end{equation}

For the other case where the maximum is attained on the right boundary, one may substitute the right boundary value into (\ref{eq:costbound3}), which, if evaluated, yields

\begin{equation}\label{eq:costbound6}
\begin{array}{l}
\displaystyle \phi_p({\bf u}_{k+1}) - \phi_p({\bf u}_{k}) \leq \\
\hspace{10mm} \displaystyle 2(\gamma_\phi - 1) \frac{ \left[ \nabla \phi_p ({\bf u}_k)^T (\bar {\bf u}_{k+1}^* - {\bf u}_k) \right]^2}{\mathop {\sum} \limits_{i_1=1}^{n_u} \mathop {\sum} \limits_{i_2=1}^{n_u} M_{\phi,i_1 i_2} | (\bar u_{k+1,i_1}^* - u_{k,i_1})(\bar u_{k+1,i_2}^* - u_{k,i_2}) |}.
\end{array}
\end{equation}

Since this is a negative quantity, it may be globally upper bounded by minimizing its numerator and maximizing its denominator. Using the same bounds as before immediately yields

\begin{equation}\label{eq:costbound7}
\displaystyle \phi_p({\bf u}_{k+1}) - \phi_p({\bf u}_{k}) \leq 2(\gamma_\phi - 1) \frac{ \delta_\phi^2}{\overline Q_\phi} < 0.
\end{equation}

Consider now the case where $\gamma_\phi = 0$, which could be the case if $\phi_p$ were linear. If this were so, then (\ref{eq:costbound3}) simplifies to

\begin{equation}\label{eq:costbound7a}
\phi_p({\bf u}_{k+1}) - \phi_p({\bf u}_{k}) \leq K_k \nabla \phi_p({\bf u}_{k})^T (\bar {\bf u}_{k+1}^* - {\bf u}_{k}) \leq - \underline K \delta_\phi < 0.
\end{equation}

\noindent However, this is just a special case of (\ref{eq:costbound5}).

To obtain a global bound that accounts for both (\ref{eq:costbound5}) and (\ref{eq:costbound7}), it is sufficient to take the maximum of the two:

\begin{equation}\label{eq:costbound8}
\displaystyle \phi_p({\bf u}_{k+1}) - \phi_p({\bf u}_{k}) \leq \mathop {\max} \left[ \underline K \left( \underline K \displaystyle \frac{\gamma_\phi \overline Q_\phi}{2} - \delta_\phi \right), 2(\gamma_\phi - 1) \frac{ \delta_\phi^2}{\overline Q_\phi} \right] < 0.
\end{equation}

As this decrease is ensured whenever the projection is feasible, and as the maximum suboptimality gap cannot be greater than $\phi_p({\bf u}_0) - \underline \phi_p$ for feasible-side iterates, it follows that the projection cannot be feasible for more experiments than the number given in (\ref{eq:maxiter}), as this would guarantee decreasing the cost past its global minimum value. \end{proof}  

\subsection{Fritz John Error at the Converged Point}

The FJ conditions that must be satisfied by an FJ point, ${\bf u}^*$, for Problem (\ref{eq:mainprob}) are \citep{John1948,Mangasarian1967}

\begin{equation}\label{eq:KKT}
\begin{array}{ll}
{g}_{p,j} ({\bf{u}}^*) \leq 0& j = 1,...,n_{g_p} \vspace{1mm}  \\
{g}_{j} ({\bf{u}}^*) \leq 0 & j = 1,...,n_g  \vspace{1mm} \\
{\bf u}^L \preceq {\bf u}^* \preceq {\bf u}^U & \vspace{1mm} \\
\mu_{p,j} g_{p,j} ({\bf{u}}^*) = 0 &  j = 1,...,n_{g_p} \vspace{1mm} \\
\mu_j g_{j} ({\bf{u}}^*) = 0 &  j = 1,...,n_{g} \vspace{1mm} \\
\zeta_i^L(u^L_i - u^*_i) = 0,\; \zeta_i^U(u^*_i - u^U_i) = 0 &  i = 1,...,n_u \vspace{1mm}\\
\nabla \mathcal{L}({\bf u}^*) = \mu_\phi \nabla \phi_p ({\bf{u}}^*) + \displaystyle \sum\limits_{j = 1}^{n_{g_p}} {\mu_{p,j} \nabla g_{p,j} ({\bf{u}}^*)} & \\
\hspace{25mm}+ \displaystyle \sum\limits_{j = 1}^{n_g} {\mu_{j} \nabla g_{j} ({\bf{u}}^*)} - \boldsymbol{\zeta}^L + \boldsymbol{\zeta}^U = {\bf{0}} & \vspace{1mm} \\
{\boldsymbol \mu}_{all} = \left[ \begin{array}{c}  \mu_\phi \\  {\boldsymbol \mu}_p \\ {\boldsymbol \mu} \\ {\boldsymbol \zeta}^L \\ {\boldsymbol \zeta}^U \end{array} \right] \neq {\bf 0},
\end{array}
\end{equation}

\noindent where $\mu_\phi \in \mathbb{R}_+$, ${\boldsymbol \mu_{p}} \in \mathbb{R}^{n_{g_p}}_+$, and ${\boldsymbol \mu} \in \mathbb{R}^{n_{g}}_+$ are the Lagrange multipliers for the cost, the experimental constraint, and the numerical constraint functions, respectively, while ${\boldsymbol \zeta}^L, {\boldsymbol \zeta}^U \in \mathbb{R}^{n_u}_+$ are the Lagrange multipliers for the lower and upper bound constraints. For compactness of notation, we denote by ${\boldsymbol \mu}_{all}$ the collective vector of multipliers. $\mathcal{L} : \mathbb{R}^{n_u} \rightarrow \mathbb{R}$ denotes the Lagrangian.

Noting that satisfaction of Conditions (\ref{eq:SCFO1i}), (\ref{eq:SCFO2i}), and (\ref{eq:SCFO3}) ensures that all of the experimental iterates remain feasible and thus satisfy the primal feasibility conditions of (\ref{eq:KKT}), we are only interested in the degree of suboptimality that is reflected by the lack of satisfaction of the equalities in (\ref{eq:KKT}), and define the \emph{FJ error metric} as the minimal sum of squared errors in the FJ equalities:

\begin{equation}\label{eq:KKTerr}
\mathcal{E}({\bf u}) = \hspace{-2mm} \mathop {\min} \limits_{ \footnotesize{\begin{array}{c} {\boldsymbol \mu}_{all} \succeq {\bf 0} \\ \| {\boldsymbol \mu}_{all} \| = 1 \end{array} }} \left( \begin{array}{l} \nabla \mathcal{L}({\bf u})^T \nabla \mathcal{L}({\bf u}) +  \displaystyle \sum\limits_{j = 1}^{n_{g_p}} \left[ \mu_{p,j} g_{p,j} ({\bf u}) \right]^2 +  \displaystyle \sum\limits_{j = 1}^{n_g} \left[ \mu_j g_{j} ({\bf u}) \right]^2 \\
+\displaystyle \sum\limits_{i = 1}^{n_u} \left[ ( \zeta^L_i (u^L_i - u_i) )^2 + ( \zeta^U_i (u_i - u^U_i) )^2  \right] \end{array} \right),
\end{equation}

\noindent subject to the additional constraint $\| {\boldsymbol \mu}_{all} \| = 1$ ($\| \cdot \|$ denoting any standard norm), which ensures ${\boldsymbol \mu}_{all} \neq {\bf 0}$ while bounding the magnitude of the multiplier vector away from zero -- otherwise, $\mathcal{E} ({\bf u})$ could be made arbitrarily small for any primally feasible ${\bf u}$ by simply choosing very small multipliers. It should be clear that $\mathcal{E} ({\bf u}^*) = 0$ since one can uniformly scale the multipliers to obtain $\| {\boldsymbol \mu}_{all} \| = 1$ without invalidating (\ref{eq:KKT}).

Denoting by ${\bf u}_\infty$ the first point where the projection becomes infeasible for some choice of ${\boldsymbol \epsilon}_{p}, {\boldsymbol \epsilon}, {\boldsymbol \delta}_{g_p}, {\boldsymbol \delta}_{g}, \delta_\phi$ -- thereby forcing the iterates to converge to ${\bf u}_\infty$ -- we now prove that the proposed scheme converges abritrarily close to an FJ point in the FJ-error sense as the projection parameters are made arbitrarily small, i.e., that $\mathcal{E} ({\bf u}_\infty) \rightarrow 0$ as ${\boldsymbol \epsilon}_{p}, {\boldsymbol \epsilon}, {\boldsymbol \delta}_{g_p}, {\boldsymbol \delta}_{g}, \delta_\phi \downarrow {\bf 0}$.

\begin{theorem}[Convergence to an FJ point in the FJ-error sense]
\label{thm:converge}
Let the project-and-filter approach of (\ref{eq:proj}) and (\ref{eq:inputfilter}) be applied at every experimental iteration where the projection is feasible, with $K_k$ chosen as the maximum value on the interval $[0,1]$ subject to the limitations of (\ref{eq:SCFO1i}), (\ref{eq:SCFO2i}), and (\ref{eq:SCFO7i}), and let ${\bf u}_{k+1} := {\bf u}_k$ if the projection is not feasible, with ${\bf u}_\infty$ denoting the point where this occurs. It follows that

\begin{equation}\label{eq:KKTerr0}
\mathop {\lim} \limits_{{\boldsymbol \epsilon}_{p}, {\boldsymbol \epsilon}, {\boldsymbol \delta}_{g_p}, {\boldsymbol \delta}_{g}, \delta_\phi \downarrow {\bf 0}} \mathcal{E}({\bf u}_\infty) = 0.
\end{equation}

\end{theorem}
\begin{proof}
The infeasibility of the projection at ${\bf u}_\infty$ indicates that the system of inequalities

\begin{equation}\label{eq:globproof1}
\begin{array}{rcl}
\left[ \begin{array}{c}
\nabla {\bf g}_p^\epsilon ({\bf u}_\infty)^T \\
\nabla {\bf g}^\epsilon ({\bf u}_\infty)^T \\
\nabla \phi_p ({\bf u}_\infty)^T
\end{array} \right] ({\bf u} - {\bf u}_\infty) & \preceq &
\left[ \begin{array}{c}
-{\boldsymbol \delta}_{g_p}^\epsilon \\
-{\boldsymbol \delta}_{g}^\epsilon \\
-\delta_\phi
\end{array} \right] \\
-{\bf u} + {\bf u}^L & \preceq & {\bf 0} \\
{\bf u} - {\bf u}^U & \preceq & {\bf 0}
\end{array}
\end{equation}

\noindent has no solution. Here, $\nabla {\bf g}_p^\epsilon$ and $\nabla {\bf g}^\epsilon$ are the Jacobian matrices of the ``$\epsilon$-active'' experimental and numerical constraint sets, respectively, with ${\boldsymbol \delta}_{g_p}^\epsilon$ and ${\boldsymbol \delta}_g^\epsilon$ the corresponding vectors of projection parameters -- note that this is simply a different way of writing the constraints of (\ref{eq:proj}). So as to lump those bound constraints that are active into the upper set of inequalities, we denote by ${\bf u}_a$ the subvector of decision variables whose lower bounds are active at ${\bf u}_\infty$, by ${\bf u}^a$ the subvector of decision variables whose upper bounds are active at ${\bf u}_\infty$, and use $(\tilde \cdot)$ to denote their complements, i.e.:

\begin{equation}\label{eq:globproof1a}
({\bf u}_{a})_\infty = ({\bf u}_a)^L, \;\; ({\bf u}^a)_\infty = ({\bf u}^{a})^U, \;\; (\tilde {\bf u}_{a})_\infty \succ (\tilde {\bf u}_a)^L, \;\; (\tilde {\bf u}^a)_\infty \prec (\tilde {\bf u}^{a})^U,
\end{equation}

\noindent which allows us to rewrite the bound constraints as

\begin{equation}\label{eq:globproof2}
\begin{array}{rcl}
-{\bf u} + {\bf u}^L & \preceq & {\bf 0} \\
{\bf u} - {\bf u}^U & \preceq & {\bf 0}
\end{array} \Leftrightarrow
\begin{array}{rcl}
-{\bf u}_a + ({\bf u}_a)_\infty & \preceq & {\bf 0} \\
{\bf u}^a - ({\bf u}^a)_\infty & \preceq & {\bf 0} \\
-\tilde {\bf u}_a + (\tilde {\bf u}_a)^L & \preceq & {\bf 0} \\
\tilde {\bf u}^a - (\tilde {\bf u}^a)^U & \preceq & {\bf 0}.
\end{array}
\end{equation}

As ${\bf u}_a$ and ${\bf u}^a$ are subvectors of ${\bf u}$, there clearly exists a (very sparse) linear mapping, denoted by ${\bf U}$, so that

\begin{equation}\label{eq:globproof3}
\begin{array}{rcl}
-{\bf u}_a + ({\bf u}_a)_\infty & \preceq & {\bf 0} \\
{\bf u}^a - ({\bf u}^a)_\infty & \preceq & {\bf 0}
\end{array} \Leftrightarrow
\begin{array}{rcl}
{\bf U} ({\bf u} - {\bf u}_\infty) & \preceq & {\bf 0}
\end{array},
\end{equation}

\noindent and which is easily shown to consist of rows where all but one of the elements are 0, and the one non-zero element is either $1$ or $-1$, depending on whether the active bound is upper or lower, respectively, with its index corresponding to the index of the active bound constraint. This now allows (\ref{eq:globproof1}) to be rewritten as

\begin{equation}\label{eq:globproof4}
\begin{array}{rcl}
\left[ \begin{array}{c}
\nabla {\bf g}_p^\epsilon ({\bf u}_\infty)^T \\
\nabla {\bf g}^\epsilon ({\bf u}_\infty)^T \\
\nabla \phi_p ({\bf u}_\infty)^T \\
{\bf U}
\end{array} \right] ({\bf u} - {\bf u}_\infty) & \preceq &
\left[ \begin{array}{c}
-{\boldsymbol \delta}_{g_p}^\epsilon \\
-{\boldsymbol \delta}_g^\epsilon \\
-\delta_\phi \\
{\bf 0}
\end{array} \right] \\
-\tilde {\bf u}_a + (\tilde {\bf u}_a)^L & \preceq & {\bf 0} \\
\tilde {\bf u}^a - (\tilde {\bf u}^a)^U & \preceq & {\bf 0}.
\end{array}
\end{equation}

It is first shown that the bound constraints corresponding to the bounds that are inactive at ${\bf u}_\infty$ may be removed from analysis. Considering the relaxed projection

\begin{equation}\label{eq:globproof4a}
\begin{array}{rcl}
\left[ \begin{array}{c}
\nabla {\bf g}_p^\epsilon ({\bf u}_\infty)^T \\
\nabla {\bf g}^\epsilon ({\bf u}_\infty)^T \\
\nabla \phi_p ({\bf u}_\infty)^T \\
{\bf U}
\end{array} \right] ({\bf u} - {\bf u}_\infty) & \preceq &
\left[ \begin{array}{c}
-{\boldsymbol \delta}_{g_p}^\epsilon \\
-{\boldsymbol \delta}_g^\epsilon \\
-\delta_\phi \\
{\bf 0}
\end{array} \right],
\end{array}
\end{equation}

\noindent we will prove that $\exists {\boldsymbol \delta}_{g_p}^L, {\boldsymbol \delta}_{g}^L, \delta_\phi^L \succ {\bf 0}$ such that, for ${\boldsymbol \delta}_{g_p} \preceq {\boldsymbol \delta}_{g_p}^L,\; {\boldsymbol \delta}_{g} \preceq {\boldsymbol \delta}_{g}^L,\; \delta_\phi \leq \delta_\phi^L$,

\begin{equation}\label{eq:globproof4b}
(\ref{eq:globproof4})\;{\rm infeasible} \Rightarrow (\ref{eq:globproof4a})\;{\rm infeasible}
\end{equation}

\noindent by proving its contrapositive, i.e.:

\begin{equation}\label{eq:globproof4c}
(\ref{eq:globproof4a})\;{\rm feasible} \Rightarrow (\ref{eq:globproof4})\;{\rm feasible}.
\end{equation}

Defining $\Delta {\bf u}_\infty = {\bf u} - {\bf u}_\infty$ and, analogously, $(\Delta \tilde {\bf u}_a)_\infty = \tilde {\bf u}_a - (\tilde {\bf u}_a)_\infty$, $(\Delta \tilde {\bf u}^a)_\infty = \tilde {\bf u}^a - (\tilde {\bf u}^a)_\infty$, let us consider the inactive bound constraints in the equivalent form

\begin{equation}\label{eq:globproof5}
\begin{array}{rcl}
-(\tilde {\bf u}_a)_\infty - (\Delta \tilde {\bf u}_a)_\infty + (\tilde {\bf u}_a)^L & \preceq & {\bf 0} \\
(\tilde {\bf u}^a)_\infty + (\Delta \tilde {\bf u}^a)_\infty - (\tilde {\bf u}^a)^U & \preceq & {\bf 0}.
\end{array}
\end{equation}

Let $\Delta {\bf u}^*_\infty$ denote a feasible solution of (\ref{eq:globproof4a}), so that

\begin{equation}\label{eq:globproof6}
\begin{array}{rcl}
\left[ \begin{array}{c}
\nabla {\bf g}_p^\epsilon ({\bf u}_\infty)^T \\
\nabla {\bf g}^\epsilon ({\bf u}_\infty)^T \\
\nabla \phi_p ({\bf u}_\infty)^T \\
{\bf U}
\end{array} \right] \Delta {\bf u}^*_\infty & \preceq &
 \left[ \begin{array}{c}
- {\boldsymbol \delta}_{g_p}^\epsilon \\
- {\boldsymbol \delta}_g^\epsilon \\
- \delta_\phi \\
{\bf 0}
\end{array} \right].
\end{array}
\end{equation}

Since

\begin{equation}\label{eq:globproof7}
\begin{array}{rcl}
\left[ \begin{array}{c}
\nabla {\bf g}_p^\epsilon ({\bf u}_\infty)^T \\
\nabla {\bf g}^\epsilon ({\bf u}_\infty)^T \\
\nabla \phi_p ({\bf u}_\infty)^T \\
{\bf U}
\end{array} \right] \alpha \Delta {\bf u}^*_\infty & \preceq &
\alpha \left[ \begin{array}{c}
- {\boldsymbol \delta}_{g_p}^\epsilon \\
- {\boldsymbol \delta}_g^\epsilon \\
- \delta_\phi \\
{\bf 0}
\end{array} \right]
\end{array}
\end{equation}

\noindent clearly holds for any $\alpha \geq 0$, we can always choose ${\boldsymbol \delta}_{g_p} := \alpha {\boldsymbol \delta}_{g_p}$, ${\boldsymbol \delta}_{g} := \alpha {\boldsymbol \delta}_{g}$, $\delta_\phi := \alpha \delta_\phi$ so that $\Delta {\bf u}_\infty := \alpha \Delta {\bf u}^*_\infty$ solves (\ref{eq:globproof4a}) for this choice of ${\boldsymbol \delta}_{g_p}, {\boldsymbol \delta}_{g}, \delta_\phi$.

Substituting this $\Delta {\bf u}_\infty$ into the bound constraints (\ref{eq:globproof5}) yields

\begin{equation}\label{eq:globproof8}
\begin{array}{rcl}
-(\tilde {\bf u}_a)_\infty - \alpha (\Delta \tilde {\bf u}_a)^*_\infty + (\tilde {\bf u}_a)^L & \preceq & {\bf 0} \\
(\tilde {\bf u}^a)_\infty + \alpha (\Delta \tilde {\bf u}^a)^*_\infty - (\tilde {\bf u}^a)^U & \preceq & {\bf 0}.
\end{array}
\end{equation}

\noindent Since $-(\tilde {\bf u}_a)_\infty + (\tilde {\bf u}_a)^L \prec {\bf 0}$ and $(\tilde {\bf u}^a)_\infty - (\tilde {\bf u}^a)^U \prec {\bf 0}$ by definition, it follows that there exists a sufficiently small $\alpha := \alpha^L > 0$ such that these constraints are satisfied. As $\Delta {\bf u}_\infty := \alpha^L \Delta {\bf u}^*_\infty$ is feasible for (\ref{eq:globproof4a}) for the choice of parameters ${\boldsymbol \delta}_{g_p} := {\boldsymbol \delta}_{g_p}^L = \alpha^L {\boldsymbol \delta}_{g_p}$, ${\boldsymbol \delta}_{g} := {\boldsymbol \delta}_{g}^L = \alpha^L {\boldsymbol \delta}_{g}$, $\delta_\phi := \delta_\phi^L = \alpha^L \delta_\phi$ and as this feasibility is retained for any smaller ${\boldsymbol \delta}_{g_p}$, ${\boldsymbol \delta}_{g}$, $\delta_\phi$ due to the constraints being less stringent, it follows that a feasible solution for (\ref{eq:globproof4}) must exist for ${\boldsymbol \delta}_{g_p} \preceq {\boldsymbol \delta}_{g_p}^L,\; {\boldsymbol \delta}_{g} \preceq {\boldsymbol \delta}_{g}^L,\; \delta_\phi \leq \delta_\phi^L$, which proves (\ref{eq:globproof4c}) and thus (\ref{eq:globproof4b}). It thus follows that (\ref{eq:globproof4a}) must be infeasible if (\ref{eq:globproof4}) is infeasible for ${\boldsymbol \delta}_{g_p} \preceq {\boldsymbol \delta}_{g_p}^L,\; {\boldsymbol \delta}_{g} \preceq {\boldsymbol \delta}_{g}^L,\; \delta_\phi \leq \delta_\phi^L$, which must occur as ${\boldsymbol \delta}_{g_p}, {\boldsymbol \delta}_{g}, \delta_\phi \downarrow {\bf 0}$.

From Gale's Theorem (see, e.g., Th. 22.1 in \cite{Rockafellar1970}), infeasibility in (\ref{eq:globproof4a}) implies that there exist coefficients $\tilde \mu_\phi \in \mathbb{R}_+$, $\tilde {\boldsymbol \mu}_{p} \in \mathbb{R}^{n_{g_p}}_+$, $\tilde {\boldsymbol \mu} \in \mathbb{R}^{n_{g}}_+$, $\tilde {\boldsymbol \zeta}^L \in \mathbb{R}^{n_u}_+$, and $\tilde {\boldsymbol \zeta}^U \in \mathbb{R}^{n_u}_+$, collectively denoted by $\tilde {\boldsymbol \mu}_{all}$, with at least one of these coefficients strictly positive, such that

\begin{equation}\label{eq:globproof9}
\nabla \mathcal{\tilde L} ({\bf u}_\infty) =  \tilde \mu_\phi \nabla \phi_p ({\bf{u}}_\infty) + \displaystyle \sum\limits_{j = 1}^{n_{g_p}}  { \tilde \mu_{p,j} \nabla g_{p,j} ({\bf{u}}_\infty)} + \displaystyle \sum\limits_{j = 1}^{n_g}  {\tilde \mu_{j} \nabla g_{j} ({\bf{u}}_\infty)} - \tilde {\boldsymbol \zeta}^L + \tilde {\boldsymbol \zeta}^U = {\bf 0},
\end{equation}

\noindent where, so as to maintain the consequence of Gale's Theorem and to exclude the elements not present in (\ref{eq:globproof4a}), we force $\tilde \mu_{p,j} = 0, \; \forall j : g_{p,j} ({\bf u}_\infty) < - \epsilon_{p,j}$, $\tilde \mu_{j} = 0, \; \forall j : g_{j} ({\bf u}_\infty) < - \epsilon_{j}$, $\tilde \zeta_{i}^L = 0, \; \forall i : u_{\infty,i} > u_{i}^L$, $\tilde \zeta_{i}^U = 0, \; \forall i : u_{\infty,i} < u_{i}^U$. Here, we have written the negative spanning implied by Gale's Theorem as an analogue to the stationarity condition in (\ref{eq:KKT}), using $\mathcal {\tilde L}$ to denote the analogue to the Lagrangian. Without loss of generality, let us, without changing notation, scale (\ref{eq:globproof9}) to force $\| \tilde {\boldsymbol \mu}_{all} \| = 1$, which is possible since $\tilde {\boldsymbol \mu}_{all}$ contains at least one strictly positive element.

It is now possible to use (\ref{eq:globproof9}) to redefine the Lagrangian gradient at ${\bf u}_\infty$ as

\begin{equation}\label{eq:globproof10}
\begin{array}{l}
\nabla \mathcal{L} ({\bf u}_\infty) = \nabla \mathcal{L} ({\bf u}_\infty) - \nabla \mathcal{\tilde L} ({\bf u}_\infty) = \\
\hspace{20mm} (\mu_\phi - \tilde \mu_\phi) \nabla \phi_p ({\bf{u}}_\infty) + \displaystyle \sum\limits_{j = 1}^{n_{g_p}}  { (\mu_{p,j} - \tilde \mu_{p,j}) \nabla g_{p,j} ({\bf{u}}_\infty)} \\
\hspace{20mm} + \displaystyle \sum\limits_{j = 1}^{n_g}  {(\mu_{j} - \tilde \mu_{j}) \nabla g_{j} ({\bf{u}}_\infty)} - ({\boldsymbol \zeta}^L - \tilde {\boldsymbol \zeta}^L) + ( {\boldsymbol \zeta}^U - \tilde {\boldsymbol \zeta}^U).
\end{array}
\end{equation}

The proof is completed by upper bounding the FJ error as follows:

$$
\begin{array}{l}
\mathop {\min} \limits_{ \footnotesize{\begin{array}{c} {\boldsymbol \mu}_{all} \succeq {\bf 0} \\ \| {\boldsymbol \mu}_{all} \| = 1 \end{array} }} \left( \begin{array}{l} \nabla \mathcal{L}({\bf u}_\infty)^T \nabla \mathcal{L}({\bf u}_\infty) +  \displaystyle \sum\limits_{j = 1}^{n_{g_p}} \left[ \mu_{p,j} g_{p,j} ({\bf u}_\infty) \right]^2 +  \displaystyle \sum\limits_{j = 1}^{n_g} \left[ \mu_j g_{j} ({\bf u}_\infty) \right]^2 \\
+\displaystyle \sum\limits_{i = 1}^{n_u} \left[ ( \zeta^L_i (u^L_i - u_{\infty,i}) )^2 + ( \zeta^U_i (u_{\infty,i} - u^U_i) )^2  \right] \end{array} \right) \\
\leq \mathop {\min} \limits_{ \footnotesize{\begin{array}{c} {\boldsymbol \mu}_{all} = \tilde {\boldsymbol \mu}_{all} \end{array} }} \left( \begin{array}{l} \nabla \mathcal{L}({\bf u}_\infty)^T \nabla \mathcal{L}({\bf u}_\infty) +  \displaystyle \sum\limits_{j = 1}^{n_{g_p}} \left[ \mu_{p,j} g_{p,j} ({\bf u}_\infty) \right]^2 +  \displaystyle \sum\limits_{j = 1}^{n_g} \left[ \mu_j g_{j} ({\bf u}_\infty) \right]^2 \\
+\displaystyle \sum\limits_{i = 1}^{n_u} \left[ ( \zeta^L_i (u^L_i - u_{\infty,i}) )^2 + ( \zeta^U_i (u_{\infty,i} - u^U_i) )^2  \right] \end{array} \right) \\
= \displaystyle \sum\limits_{j = 1}^{n_{g_p}} \left[ \tilde \mu_{p,j} g_{p,j} ({\bf u}_\infty) \right]^2 +  \displaystyle \sum\limits_{j = 1}^{n_g} \left[ \tilde \mu_j g_{j} ({\bf u}_\infty) \right]^2  + \sum\limits_{i = 1}^{n_u} \left[ ( \tilde \zeta^L_i (u^L_i - u_{\infty,i}) )^2 + ( \tilde \zeta^U_i (u_{\infty,i} - u^U_i) )^2  \right] \\
= \displaystyle \sum\limits_{j = 1}^{n_{g_p}} \left[ \tilde \mu_{p,j} g_{p,j} ({\bf u}_\infty) \right]^2 +  \displaystyle \sum\limits_{j = 1}^{n_g} \left[ \tilde \mu_j g_{j} ({\bf u}_\infty) \right]^2 \\
\leq \displaystyle \sum\limits_{j = 1}^{n_{g_p}} \left[ \tilde \mu_{p,j} \epsilon_{p,j} \right]^2 +  \displaystyle \sum\limits_{j = 1}^{n_g} \left[ \tilde \mu_j \epsilon_j \right]^2,
\end{array}
$$

\noindent with the following step-by-step justifications:

\begin{enumerate}
\item Setting the Lagrange multipliers to the $(\tilde \cdot)$ analogues is equivalent to taking the minimum over a smaller feasible set (${\boldsymbol \mu}_{all} = \tilde {\boldsymbol \mu}_{all}, \| {\boldsymbol \mu}_{all} \| = 1$ as opposed to ${\boldsymbol \mu}_{all} \succeq {\bf 0}, \| {\boldsymbol \mu}_{all} \| = 1$), which can only increase the FJ error.
\item Evaluated for this choice of Lagrange multipliers, the $\nabla \mathcal{L} ({\bf u}_\infty)$ term is clearly {\bf 0} (see (\ref{eq:globproof10})).
\item The complementary slackness terms with respect to the bound constraints are all 0 since (a) $\tilde \zeta_i^L$ and $\tilde \zeta_i^U$ are 0 by definition for any bound constraints that are not active, or (b) the $u_i^L - u_{\infty,i}$ or $u_{\infty,i} - u_i^U$ terms are 0 for those that are active.
\item The $\tilde \mu_{p,j}$ and $\tilde \mu_{j}$ coefficients are 0 by definition for any inequality constraints that are not ``$\epsilon$-active''. For those that are, the upper bounds $[g_{p,j} ({\bf u}_\infty)]^2 \leq \epsilon_{p,j}^2$ and $[g_{j} ({\bf u}_\infty)]^2 \leq \epsilon_{j}^2$ follow directly from the $\epsilon$-active property.
\end{enumerate}

We have thus derived an upper bound on the FJ error that is valid for ${\boldsymbol \delta}_{g_p} \preceq {\boldsymbol \delta}_{g_p}^{L}$, ${\boldsymbol \delta}_g \preceq {\boldsymbol \delta}_g^{L}$, and $\delta_\phi \leq \delta_\phi^{L}$ -- which will become the case as ${\boldsymbol \delta}_{g_p}, {\boldsymbol \delta}_{g}, \delta_\phi \downarrow {\bf 0}$ -- and that will go to 0 as ${\boldsymbol \epsilon}_p, {\boldsymbol \epsilon} \downarrow {\bf 0}$. As the FJ error is bounded from below by 0, it follows that the error go to 0 in the limit as well. \end{proof}

This result has a very simple geometric interpretation, in that the projection becomes infeasible when there is no longer a direction, ${\bf u} - {\bf u}_\infty$, that is both locally cost decreasing and locally feasible with some approximation error. As this approximation error goes to 0, the infeasibility of the projection implies the lack of existence of a locally cost decreasing and feasible direction, which is nothing else than the geometric conditions for an FJ point. 

\section{An Adaptive Choice of Projection Parameters}
\label{sec:implement2}

While one may choose the projection parameters prior to applying the proposed method and keep them the same for every experimental iteration, it is not clear what single choice of the parameters would be best. As has been verified previously by the authors \cite{Bunin:12b}, a reasonable strategy is to use larger values in the earlier experiments, as this tends to promote greater decreases in the cost while staying far away from the constraints, and to use lower values once the projection becomes infeasible, since it is only for very small values that the results of Theorem \ref{thm:converge} become meaningful with respect to optimality. In light of this we propose a scheme to adaptively choose the parameters before every experimental iteration:

\vspace{2mm}
\noindent {\bf{Initialization -- Done Only Once Prior to First Experimental Iteration}}
\vspace{2mm}

\begin{enumerate}
\item Define $\overline \epsilon_{p,j} = \overline \delta_{g_p,j} \approx - \mathop {\min} \limits_{{\bf u} \in \mathcal{I}} g_{p,j} ({\bf u})$, $\overline \epsilon_{j} = \overline \delta_{g,j} \approx - \mathop {\min} \limits_{{\bf u} \in \mathcal{I}} g_{j} ({\bf u})$, and $\overline \delta_\phi \approx \phi_p ({\bf u}_0) - \mathop {\min} \limits_{{\bf u} \in \mathcal{I}} \phi_p ({\bf u})$. 
\end{enumerate}

\vspace{2mm}
\noindent {\bf{Search for a Feasible Projection -- Prior to Each Experimental Iteration}}
\vspace{2mm}

\begin{enumerate}
\setcounter{enumi}{1}
\item Set ${\boldsymbol \epsilon}_p := \overline {\boldsymbol \epsilon}_p$, ${\boldsymbol \epsilon} := \overline {\boldsymbol \epsilon}$, $\boldsymbol{\delta}_{g_p} := \boldsymbol{\overline \delta}_{g_p}$, $\boldsymbol{\delta}_{g} := \boldsymbol{\overline \delta}_{g}$, and $\delta_\phi := \overline \delta_\phi$.
\item Check the feasibility of (\ref{eq:proj}) for the given choice of ${\boldsymbol \epsilon}_{p}, {\boldsymbol \epsilon}, {\boldsymbol \delta}_{g_p}, {\boldsymbol \delta}_{g}, \delta_\phi$ by solving a linear programming feasibility problem. If no solution is found and $\delta_\phi \geq \overline \delta_\phi / 2^{10}$, set ${\boldsymbol \epsilon}_p := {\boldsymbol \epsilon}_p/2$, ${\boldsymbol \epsilon} := {\boldsymbol \epsilon}/2$, $\boldsymbol{\delta}_{g_p} := \boldsymbol{\delta}_{g_p}/2$, $\boldsymbol{\delta}_{g} := \boldsymbol{\delta}_{g}/2$, $\delta_\phi := \delta_\phi/2$ and repeat this step. If no solution is found and $\delta_\phi < \overline \delta_\phi / 2^{10}$, go to Step 5. Otherwise, if a feasible solution is found, proceed to Step 4.

\item Solve (\ref{eq:proj}) with the resulting ${\boldsymbol \epsilon}_{p}, {\boldsymbol \epsilon}, {\boldsymbol \delta}_{g_p}, {\boldsymbol \delta}_{g}, \delta_\phi$ to obtain $\bar {\bf u}_{k+1}^*$.

\end{enumerate}

\vspace{2mm}
\noindent {\bf{Termination -- Declared Convergence to an FJ Point}}
\vspace{2mm}

\begin{enumerate}
\setcounter{enumi}{4}
\item Set $\bar {\bf u}_{k+1}^* := {\bf u}_{k}$.
\end{enumerate}

\vspace{5mm}

The first step of this scheme acts as a sort of scaling by setting the maximum projection parameters as being approximately equal to the sizes of the respective function ranges. These approximations, in many cases based on some sort of engineering knowledge, may be rather brute, since what essentially needs to be correct is the order of magnitude. Clearly, using such an initial setting (Step 2) will always lead the projection to attempt to decrease the linear approximations of both the cost and the constraints by the maximum amount possible, i.e., attempting to obtain large decreases in the cost while staying away from all of the problem constraints. While desired, this will almost certainly be impossible for most constrained problems, thus leading to the projection parameters being halved until the projection becomes feasible. If this reduction has taken place a certain number of times (here, we choose 10) and the projection is still infeasible, then the FJ error is considered to be sufficiently small so as to declare convergence to an FJ point. While one could argue against the rigor of this approach -- by, for example, providing examples where the FJ error is still large even after $\overline \delta_\phi$ and the other parameters have been cut by $2^{10}$ -- we believe such a scheme to be sufficient for most problems.

\section{Examples}
\label{sec:example}

We investigate the potential strengths and weaknesses of the proposed framework by considering two constructed example problems, where all experimental functions are given analytically and numerical evaluations are used in place of actual experiments so as to make performance analysis possible and computationally cheap.

\subsection{Minimizing an Experimental Function over an Experimentally Constrained Feasible Region}

Consider the following experimental optimization problem:

\begin{equation}\label{eq:exprob}
\begin{array}{rl}
\mathop {{\rm{minimize}}}\limits_{u_1,u_2} & \phi_{p}({\bf{u}}) := (u_1-0.5)^2 + (u_2-0.4)^2 \\
{\rm{subject}}\hspace{1mm}{\rm{to}} & g_{p,1}({\bf{u}}) := -6u^2_1 - 3.5u_1 + u_2 -0.6 \le 0 \vspace{1mm} \\
 & g_{p,2}({\bf{u}}) := 2u^2_1 + 0.5u_1 + u_2 -0.75 \le 0 \vspace{1mm} \\
 & g_{1}({\bf{u}}) := -u^2_1 - (u_2-0.15)^2 +0.01 \le 0 \vspace{1mm} \\
& -0.5 \leq u_1 \leq 0.5 \\
& 0 \leq u_2 \leq 0.8,
\end{array}
\end{equation}

\noindent with the initial experiment at ${\bf u}_0 := (-0.45,0.05)$. The following choice of Lipschitz constants is valid (i.e., it satisfies (\ref{eq:lipcon}) and (\ref{eq:lipcon2})) and is used to enforce (\ref{eq:SCFO1i}) and (\ref{eq:SCFO7i}):

\begin{equation}\label{eq:exlip}
\begin{array}{ll}
\kappa_{p,11} := 10, & \kappa_{p,12} := 2 \\
\kappa_{p,21} := 3, & \kappa_{p,22} := 2 \\
M_{\phi,11} := 3, & M_{\phi,12} := 1 \\
M_{\phi,21} := 1, & M_{\phi,22} := 3.
\end{array}
\end{equation}

The upper bounds on the projection parameters are chosen as $\overline \epsilon_{p,1} := \overline \delta_{g_p,1} := 4$, $\overline \epsilon_{p,2} := \overline \delta_{g_p,2} := 2$, $\overline \epsilon_{1} := \overline \delta_{g,1} := 1$, and $\overline \delta_\phi := 1$. Since the choice of the optimization target ${\bf u}_{k+1}^*$ does not affect the key properties of the method, we simply set, somewhat arbitrarily, ${\bf u}_{k+1}^* := (0,0.4)$ (i.e., the center of $\mathcal{I}$) for all experimental iterations.

The chain of experiments generated by the project-and-filter approach with the adaptive choice of projection parameters is shown in Fig. \ref{fig:ex}, from where we observe that enforcing Conditions (\ref{eq:SCFO1})-(\ref{eq:SCFO7}) via the project-and-filter approach does indeed lead to monotonically improving experiments that never leave the safe (feasible) region while converging extremely close to a local minimum. Note that, while the theory developed does not ensure that the FJ point be a local minimum -- it could, in principle, be a maximum or a saddle point -- the descent nature of the algorithm tends to avoid those FJ points that are not local minima as such points are innately unstable. In this example, there is a second, albeit unstable, FJ point at ${\bf u} = (-0.09,0.11)$, which the algorithm clearly circumvents.

\begin{figure}
\begin{center}
\includegraphics[width=12cm]{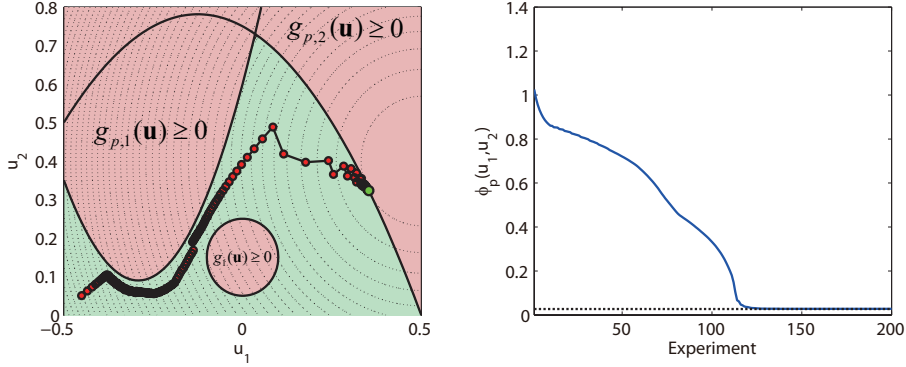}
\caption{Chain of experiments (red points) generated by applying the proposed methodology to Problem (\ref{eq:exprob}). The green point denotes the only local minimum. The dotted lines on the left plot denote the contours of the cost function, while the constant dotted line on the right denotes the cost value at the minimum.}
\label{fig:ex}
\end{center}
\end{figure}

Problem (\ref{eq:exprob}) is thus a good example of a problem for which the proposed framework is extremely appropriate, since the chain of experiments safely converges to the minimum despite the presence of experimental constraints and a feasible set that is not ``nice''.

\subsection{Minimizing the Experimental Rosenbrock Function Over a Unit Box}

To give an example of a problem for which the proposed framework may not be appropriate, we consider minimizing the Rosenbrock function \citep{Rosenbrock1960} subject to bound constraints:

\begin{equation}\label{eq:rosen}
\begin{array}{rl}
\mathop {{\rm{minimize}}}\limits_{u_1,u_2} & \phi_{p}({\bf{u}}) := (1 - u_1)^2 + 100(u_2-u_1^2)^2 \\
{\rm{subject}}\hspace{1mm}{\rm{to}} & 0 \leq u_1 \leq 1 \vspace{1mm} \\
 & 0 \leq u_2 \leq 1, 
\end{array}
\end{equation}

\noindent with the initial experiment taken as ${\bf u}_0 := (0,0)$. The required Lipschitz constants are set as

\begin{equation}\label{eq:exlip2}
\begin{array}{ll}
M_{\phi,11} := 1500, & M_{\phi,12} := 500 \\
M_{\phi,21} := 500, & M_{\phi,22} := 300,
\end{array}
\end{equation}

\noindent with $\overline \delta_\phi$ chosen as 1. For simplicity, we set the target ${\bf u}_{k+1}^*$ as the true optimum, ${\bf u}^* = (1,1)$.

A chain of 5,000 experiments is simulated and is shown to converge monotonically to the optimum as expected (Fig. \ref{fig:rosenex}). As this problem provides a good example of being able to achieve significant reductions in cost (and FJ error) without converging completely to the optimum, the termination criterion of the adaptive algorithm of Section 6 is modified and the versions $\delta_\phi < \overline \delta_\phi/2^{20}$, $\delta_\phi < \overline \delta_\phi/2^{10}$, and $\delta_\phi < \overline \delta_\phi/2^{5}$ are all tested, with the FJ errors of $3.66 \cdot 10^{-4}$, $8.8 \cdot 10^{-3}$, and $1.23 \cdot 10^{-1}$ attained at the final point for the three cases, respectively -- the FJ error at the initial experiment being equal to 4.

\begin{figure}
\begin{center}
\includegraphics[width=12cm]{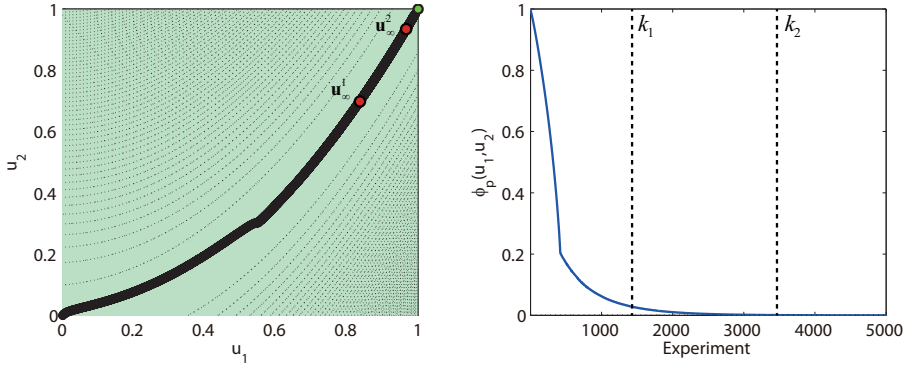}
\caption{Chain of experiments generated by applying the proposed methodology to Problem (\ref{eq:rosen}) with the modified termination criterion of $\delta_\phi < \overline \delta_\phi/2^{20}$. The effect of using different termination criteria is also shown: the algorithm converges to ${\bf u}_{\infty}^1$ in $k_1$ experiments when $\delta_\phi < \overline \delta_\phi/2^5$ is used as the criterion, and to ${\bf u}_{\infty}^2$ in $k_2$ experiments when $\delta_\phi < \overline \delta_\phi/2^{10}$ is used.}
\label{fig:rosenex}
\end{center}
\end{figure}

From a purely theoretical perspective, the proposed framework is no less successful in this example than in the one prior. However, the number of experiments needed to obtain good cost reductions (several hundred) and to ultimately reach the optimal cost value (several thousand) is unlikely to be acceptable in practice for a problem with only two degrees of freedom and, presumably, expensive experiments. A practitioner could reasonably argue that a simple 2-factorial experiment design would, though lacking theoretical rigor, find the optimum in just four experiments for this problem. Alternatively, many of the methods discussed in Section 2 that are provably convergent for bound constrained problems could also be used, and would likely require fewer experiments. When the performance of the algorithm is judged purely by the number of experiments needed to find the optimum, these arguments are largely valid and we concede that the proposed framework, whose major strength is the handling of experimental constraints, is a poor choice.

If, however, the performance is judged not only by the number of experiments but also by their suboptimality, then there is a silver lining. The pathology of the observed slow convergence should be clear -- the gradients of the cost function along the convergence path are very small compared to the Lipschitz constants of the cost function derivatives, leading to extremely small values of $K_k$ by Condition (\ref{eq:SCFO7i}) and thus to very small steps. While debilitating in terms of convergence speed, it is, at the same time, this conservatism that guarantees that the cost improve monotonically from experiment to experiment for a function where this is very difficult due to extreme nonlinear behavior. In fact, if one were to bypass this guarantee and allow larger steps, it would be very easy to go into regions of the decision-variable space where the cost function climbed drastically, which could be reflected by major economic losses in practice. From this point of view, the same 2-factorial design that would find the optimum in only four experiments would also generate two \emph{extremely} suboptimal experiments at ${\bf u} = (0,1)$ and ${\bf u} = (1,0)$, the losses from which may be so astronomical as to make smooth convergence, even if in a few thousand iterations, preferable. Such considerations are, of course, application-dependent.

\section{Concluding Remarks}
\label{sec:conclusion}

Sufficient conditions for feasible-side global convergence to a Fritz John (FJ) stationary point of an experimental optimization problem have been proposed, and it has been proven that one can converge arbitrarily close to an FJ point without ever violating the experimental inequality constraints of the problem -- the latter point being, perhaps, the key requirement that sets experimental optimization apart from other (numerical) optimization contexts. These results are very promising as they represent, to the best of our knowledge, the first theoretical tool capable of guaranteeing this sort of behavior in a context where much of the methodology has traditionally been \emph{ad hoc} in nature.

However, one cannot avoid noticing that the conditions, while having an implementable form (via the project-and-filter approach), are nevertheless conceptual in nature. A quick inspection of (\ref{eq:SCFO1})-(\ref{eq:SCFO7}) makes this clear since to enforce them one requires:

\begin{itemize}
\item the exact values of $g_{p,j} ({\bf u}_k)$, which will not be available in most experimental settings due to noise (measurement) errors,
\item the Lipschitz constants $\kappa_{p,ji}$ and $M_{\phi, i_1 i_2}$, which will not, in most cases, be known since the experimental functions themselves are assumed unknown,
\item the gradients $\nabla g_{p,j} ({\bf u}_k)$ and $\nabla\phi_p ({\bf u}_k)$, which are also unknown.
\end{itemize}

\noindent While these problems may seem daunting, the good news is that one can nevertheless propose strategies that attempt to fill these ``knowledge gaps'' with data-driven estimations, and which may then be made robust by considering the uncertainty of the constraint/Lipschitz/gradient estimates. We refer the interested reader to the supplementary document \cite{Bunin:SCFOImp} for a comprehensive discussion that essentially addresses all of the issues necessary to make the theory discussed here experimentally viable.

Finally, we want to conclude by noting that an open-source experimental optimization solver incorporating the ideas of this paper and the supplement \cite{Bunin:SCFOImp} is already available \cite{SCFOug}. Successful results for several simulated and experimental \cite{Bunin2013p} problems have been obtained, with more applications planned in the future.  

\bibliography{genbib}
\bibliographystyle{siam}

\end{document}